\documentclass{amsart}
\usepackage[top=1.4in, bottom=1.2in, left=1.2in, right=1.2in]{geometry}
\usepackage{xcolor}
\usepackage{fancyhdr}
\usepackage{enumerate}
\usepackage{comment}
\usepackage{url}
\usepackage{graphicx}
\usepackage{subfigure}
\usepackage{amsmath,amsfonts,amssymb,amsthm}
\usepackage{mathtools}
\usepackage{thmtools}
\usepackage{stmaryrd}
\usepackage{mathrsfs}
\usepackage{upgreek}
\usepackage[all,arc]{xy}

\usepackage[T1]{fontenc}
\usepackage[sc,osf]{mathpazo}

\usepackage{tikz}
\usepackage{tikz-cd}
\usetikzlibrary{matrix}

\usetikzlibrary{external}
\tikzcdset{every diagram/.append style={/tikz/external/export=false}}

\usepackage[colorlinks=true,hypertexnames=false]{hyperref}

\newtheorem{thm}{Theorem}[section]
\newtheorem{cor}[thm]{Corollary}
\newtheorem{prop}[thm]{Proposition}
\newtheorem{lem}[thm]{Lemma}

\theoremstyle{definition}
\newtheorem{defn}[thm]{Definition}

\newtheorem{con}[thm]{Construction}
\newtheorem{exmp}[thm]{Example}

\theoremstyle{remark}
\newtheorem{rem}[thm]{Remark}
\newtheorem{rems}[thm]{Remarks}

\newcommand{\Z}{\mathbb{Z}}
\newcommand{\RP}{\mathbb{RP}}
\newcommand{\RR}{\mathbb{R}}
\newcommand{\FF}{\mathbb{F}}
\newcommand{\gl}{\mathfrak{gl}}
\newcommand{\lrb}[1]{\left\llbracket #1 \right\rrbracket}
\newcommand{\slide}{\mathrm{slide}}
\newcommand{\tw}{\mathrm{tw}_+}
\newcommand{\one}{\mathbf{1}}
\DeclareMathOperator{\idd}{id}
\newcommand{\FoamN}{\mathbf{Foam}_N}

\input{figures.tex}

\title{The flip symmetry on Khovanov--Rozansky homology}
\author{Hongjian Yang}
\email{yhj@stanford.edu}
\address{Department of Mathematics, Stanford University, Stanford, CA 94305}

\date{\today}

\begin{document}

\begin{abstract}
    The flip symmetry on link diagrams induces an involution on Khovanov--Rozansky $\gl_N$ homology. We prove that this involution is the identity over $\Z$, generalizing a previous result of Chen and the author. The key ingredients of the proof are (1) a homotopy perturbation argument via a detailed study of the fork twist, which allows us to reduce the computation to planar $\gl_N$ webs, and (2) the computation of the flip map for planar $\gl_N$ webs via diagrammatics of Soergel bimodules. The latter computation can also be interpreted as a naturality result for the half twist action on type $A$ Soergel bimodules, which might be of independent interest.
\end{abstract}

\maketitle
\setcounter{tocdepth}{1}
\tableofcontents

\vspace{-5ex}
\begin{center}
    \includegraphics[width=0.15\linewidth]{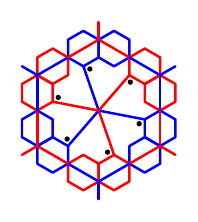}
\end{center}

\setlength{\parskip}{0.25\baselineskip}

\section{Introduction}

A continuing trend in developing new homological invariants in low-dimensional topology is to incorporate \textit{symmetries}, both those rooted in the theories themselves and those arising from symmetries of the topological objects. This has led to fruitful results in gauge theory and Floer homology, from Furuta's proof of the $10/8$-theorem \cite{furuta2001monopole} to Manolescu's disproof of the Triangulation Conjecture \cite{manolescu2016pin}. More recent progress includes the study of corks \cite{dai2022corks} and real Seiberg--Witten theory \cite{konno2021involutions,li2022monopole,konno2023involutions,miyazawa2023gauge}. Most related to the subject of the present paper is involutive Heegaard Floer homology defined by Hendricks and Manolescu \cite{hendricks2017involutive}, utilizing the conjugation symmetry on Heegaard diagrams of $3$-manifolds. 

In analogy with the $\iota$ map in Heegaard Floer theory, one can define an involution on Khovanov homology \cite{khovanov2000categorification}, using the \textit{flip symmetry} on link diagrams. To our knowledge, Viro first observed the flip symmetry, which was later exploited by Lipshitz and Sarkar \cite{lipshitz2014khovanov} to prove that their Khovanov stable homotopy type is independent of the choice of ladybug matchings. In previous work \cite{chen2025flip}, D. Chen and the author proved that this involution is the identity on Khovanov homology. Similar techniques have been used to compare the involutions on Khovanov homology arising from two types of symmetric diagrams of a strongly invertible knot \cite[Section 6]{chen2025flip} \cite{kim2026equivariant} and to refine the functoriality of Khovanov homology in $\RP^3$ \cite{ren2025intrinsic}.

The main theorem of \cite{chen2025flip} was stated for non-equivariant Khovanov homology over $\FF_2$, arguably the most vanilla member in the family of type $A$ link homology theories. There has been growing interest in studying the structural properties of this whole family \cite{khovanov2008matrix,mackaay2009sl,Rasmussen2006SomeDO,queffelec2016sln}, exploring how these theories parallel and depart from Khovanov homology. In this paper, we study the flip map on (uncolored, equivariant) Khovanov--Rozansky $\gl_N$ link homology \cite{khovanov2008matrix} over the integers, generalizing the previous result from \cite{chen2025flip}. 

Throughout the paper, fix an integer $N\ge 2$ and let $\lrb{\cdot}$ denote the (uncolored) equivariant $\gl_N$ tangle invariant (with $\Z$ coefficients) in the sense of \cite{ehrig2018functoriality}, which can be thought of as the universal version of $\gl_N$ link homology theories. Before stating the main theorem, we briefly collect constructions relevant to the flip symmetry; for details, see Section~\ref{sec:bg flip}. 

Given a link diagram $D$ for a link $K$, its \textit{flip diagram} $D^*$ is a diagram of the same link obtained by rotating $K$ in $\RR^3$ around an axis in the plane by $\pi$ and projecting it onto the same plane. The diagrams $D$ and $D^*$ are related by the trace cobordism of the rotation, which we call the \textit{flip cobordism}. The flip cobordism induces a chain homotopy equivalence \[\rho\colon\lrb{D}\to\lrb{D^*},\]well-defined up to homotopy. There is another natural identification \[\eta\colon\lrb{D}\to\lrb{D^*}\]constructed from the identification of the crossings of $D$ and $D^*$ via a vertex-wise reflection construction with sign corrections. 

The \textit{flip map} \[\kappa\coloneqq\eta^{-1}\circ\rho\colon\lrb{D}\to\lrb{D}\]is an analog of the Heegaard Floer $\iota$ map in quantum link homologies. It measures the difference between the topological identification $\rho$ and the combinatorial identification $\eta$.

Our main theorem asserts that the flip map is trivial for (uncolored, equivariant) $\gl_N$ link homology over $\Z$, generalizing \cite[Theorem 1.4]{chen2025flip}.

\begin{thm}\label{thm:main}
    Let $K$ be a link in $\RR^3$, and let $D$ be a diagram of $K$. The flip map \[\kappa\colon\lrb{D}\to\lrb{D}\]is chain homotopic to the identity map.
\end{thm}

As a corollary, we relate two involutions on link homology of a strongly invertible knot, generalizing \cite[Theorem 1.9]{chen2025flip}. Recall that there are two types of symmetric diagrams for a strongly invertible knot: the \textit{intravergent diagram}, where the axis is perpendicular to the projection plane, and the \textit{transvergent diagram}, where the axis is on the projection plane. The following corollary asserts that these two symmetric diagrams yield the same involution on $\gl_N$ link homology. The proof is essentially the same as \cite[Theorem 1.9]{chen2025flip}: it follows from the triviality of the flip map. To state the result over the integers, we need to take into account the signs that arise from orientation reversal; see Section~\ref{sec:strongly invertible}.

\begin{cor}\label{cor:strongly invertible}
    Let $(K,\tau)$ be a strongly invertible knot. Let $D_t$ and $D_i$ be a transvergent diagram and an intravergent diagram of $(K,\tau)$, respectively. The involution $\tau$ induces two chain maps
    \[\tau_t \colon \lrb{D_t} \to \lrb{D_t} \  \text{and} \  \tau_i \colon \lrb{D_i} \to \lrb{D_i}\]
    by applying the symmetry on each resolution, with the signs specified in Definition~\ref{defn:strong involutions}. Let $\varphi \colon \lrb{D_t} \to \lrb{D_i}$ be the chain homotopy equivalence induced by an isotopy from $D_t$ to $D_i$ that is equivariant up to isotopy in the sense of Section~\ref{sec:strongly invertible}. Then the following diagram commutes up to homotopy:
    \[
        \tikzexternaldisable
        {\begin{tikzcd}
            \lrb{D_t} \arrow[r, "\tau_t"] \arrow[d, "\varphi"] & \lrb{D_t} \arrow[d, "\varphi"] \\
            \lrb{D_i} \arrow[r, "\tau_i"] & \lrb{D_i}
        \end{tikzcd}}.
        \tikzexternalenable
    \]
\end{cor}

To prove Theorem~\ref{thm:main}, we first take $D$ to be the closure of a coherently oriented braid $\beta$ with $n$ crossings; every link admits such a diagram, and the flip map is independent of the choice of diagrams up to homotopy, cf. Proposition~\ref{prop:flip invariance}. There is a $\{0,1\}^n$-grading on both $\lrb{D}$ and $\lrb{D^*}$ given by remembering the $0$ or $1$ resolution at each crossing, which we refer to as the \textit{cubical grading}. The map $\eta$ preserves the cubical grading by construction. The proof then divides into two steps: we first construct a chain map $\rho'$ in the chain homotopy class of $\rho$ that preserves the cubical grading and then prove that $\rho'=\eta$.

The construction of $\rho'$ is similar to the main arguments in \cite{chen2025flip}. We introduce canceling pairs of half twists to realize the flip cobordism, apply a homotopy perturbation argument to obtain a filtered map, and conclude that the perturbed map preserves the cubical grading since it preserves the homological degree, which reduces the computation to the planar case. Compared to \cite{chen2025flip}, the most important technical difference is that the appearance of thick ($2$-labeled) edges in the cube of resolutions invalidates the cup-sliding trick from \cite{rozansky2010categorification}. Instead, we analyze in more detail the atomic component of the flip cobordism---the fork twist on two strands---to carry out the homotopy perturbation argument. Moreover, we also resolve the sign ambiguity that appears in the homotopy perturbation argument, allowing us to work over the integers.

The proof of $\rho'=\eta$, on the other hand, is more interesting. For $\mathfrak{sl}_2$ Khovanov homology that we studied in \cite{chen2025flip}, each resolution of a link diagram is a crossingless diagram of an unlink, for which the flip map is easy to compute by hand. In the (uncolored) $\gl_N$ homology case, resolutions of a link diagram can be arbitrary planar $\gl_N$ webs with $1$ and $2$ labels, and the treatment becomes significantly more sophisticated. The flip cobordisms on planar $\gl_N$ webs should be thought of as induced by a foam in $I\times\RR^3$. While a topological argument we explain in Remark~\ref{rem:flip topological argument} seems to be convincing, it would require certain functoriality for tangled (spatial) $\gl_N$ webs and foams, which to our knowledge is currently unavailable.

What we do instead is to resolve this functoriality issue in a restricted but sufficient case. When a planar $\gl_N$ web arises as a resolution of an uncolored coherently oriented braid closure, its $\gl_N$ evaluation can be generated by certain foams supporting the MOY calculus \cite{mackaay2009sl}. We prove that these elementary foams commute with the flip cobordism up to an explicit sign; see Proposition~\ref{prop:flip elementary gln foam}. This sign is given by the same weighted Euler characteristic used in the definition of $\eta$, which leads to $\rho'=\eta$. The proof relies on the slide homotopies computed in \cite{stroppel2024braiding} and the twist homotopies computed in this paper, all via diagrammatic calculations of Soergel bimodules originally developed in \cite{elias2010diagrammatics}. Our computation can also be interpreted in the language of Soergel bimodules; see Remark~\ref{rem: half twists on Soergel mod}. Alternatively, it should be controlled by the internal braid group action of the half twist on categorified quantum groups \cite{Abram2022CategorificationOT} via categorified skew Howe duality \cite{queffelec2016sln}.

\subsection*{Further directions}The present paper studies the flip map on uncolored $\gl_N$ homology, and a natural next step is to study the colored version. We believe that the argument in Section~\ref{sec: flip n strand} should remain valid with mild modifications, but the computation for planar webs might be substantially heavier. It would then be interesting to see if our construction is related to other actions on colored $\gl_N$ homology, especially the involution defined in \cite{Bodish2024SpinLH}. 

\subsection*{Organization of the paper}We review necessary background in Section~\ref{sec:bg}. We introduce the flip symmetry and define the flip map in Section~\ref{sec:bg flip}. In Section~\ref{sec:flip 2 strands}, we study the fork twists, the flip cobordism on $2$-strand tangles. This serves as the atomic ingredient of the homotopy perturbation argument we carry out in Section~\ref{sec: flip n strand}, where we construct the perturbed map $\rho'$ and hence reduce the computation to planar webs. In Section~\ref{sec:flip planar web}, we explain how to flip planar webs and complete the proof of Theorem~\ref{thm:main}. In Section~\ref{sec:strongly invertible}, we define the involutions in Corollary~\ref{cor:strongly invertible} over $\Z$ and prove the corollary. Finally, Section~\ref{sec:appendix} includes the postponed proof of Proposition~\ref{prop:flip elementary gln foam}.

\subsection*{AI disclosure}AI helped fix sign issues in the definition of the $\eta$ map in Section~\ref{sec:bg flip} and the involutions in Section~\ref{sec:strongly invertible}, draw pictures, and improve the writing quality of this manuscript. 

\subsection*{Acknowledgments}This work is a natural continuation of the previous collaboration \cite{chen2025flip} with Daren Chen, and I would like to thank him for many inspiring conversations. I would also like to thank Aaron Lauda, Qianhe Qin, Hoel Queffelec, Qiuyu Ren, and Paul Wedrich for helpful discussions, Zhenkun Li for pointing out an issue in the previous version of this draft, and my advisor Ciprian Manolescu for his continued guidance and encouragement. I am also grateful to Catharina Stroppel and Paul Wedrich for allowing me to use the TikZ package from \cite{stroppel2024braiding} to depict webs and Soergel calculus. This work was partially supported by the Simons Collaboration grant on New Structures in Low-Dimensional Topology and a Simons Dissertation Fellowship.

\section{Background}\label{sec:bg}

\subsection{Foams and $\gl_N$ link homology}\label{sec:bg gln homology}

We briefly review the construction of the (equivariant, uncolored) $\gl_N$ link homology, following \cite{wang2024n}. It takes the form of a chain complex $\lrb{D}$ of graded modules over a ground ring $R$, assigned to each link diagram $D$, assembled into a canopolis in the sense of Bar-Natan \cite{bar2005khovanov}. We only record the structural features of the theory that will be used in this paper; for a detailed treatment, we refer the reader to \cite{wang2024n,ehrig2018functoriality}.

\subsubsection*{Webs}A (planar) \textit{$\gl_N$ web} is an oriented trivalent graph properly embedded in the plane, whose edges are labeled by elements of $\{0,1,\dots,N\}$ subject to a flow condition at each vertex; self-loops, multiple edges, and circles without vertices are allowed. Throughout this paper we only encounter webs whose edges are labeled by $1$ and $2$, and we sometimes call a $2$-labeled edge a \textit{thick edge}. There are two local forms of a web near a vertex: two $1$-labeled edges merge into a $2$-labeled edge, or a $2$-labeled edge splits into two $1$-labeled edges. 

\subsubsection*{Foams}A \textit{$\gl_N$ foam} is a singular cobordism between $\gl_N$ webs. More precisely, it is a compact $2$-dimensional CW complex between $\gl_N$ webs, properly embedded in $\RR^2\times I$, such that every interior point has a neighborhood homeomorphic to one of the following three local models: an open set in $\RR^2$, the trivalent singularity (i.e., the letter Y) times $[0,1]$, or the cone over the one-skeleton of a tetrahedron. The connected components of the sets of points of these three types are called \textit{facets}, \textit{seams}, and \textit{singular vertices}, respectively. Each facet $f$ is oriented and labeled by an element $\ell(f)$ in $\{0,1,\dots,N\}$, and may carry decorations by symmetric polynomials in $\ell(f)$ variables; each seam is oriented. These data are subject to certain flow conditions and orientation compatibility conditions along seams and singular vertices. For a complete definition of $\gl_N$ foams, we refer the reader to \cite[Definition 2.8]{ehrig2018functoriality}. Throughout the paper, we read foams from bottom to top.

\subsubsection*{The Robert--Wagner foam evaluation}

A \textit{closed $\gl_N$ foam} is a $\gl_N$ foam from the empty web to itself. For a closed foam $F$, Robert and Wagner \cite{robert2020closed} define an \textit{evaluation} $\langle F\rangle$ in the ground ring $R$, which we now recall. We will usually take the universal example $R=\Z[X_1,\dots,X_N]$; one can define the theory for other rings by extension of scalars.  These formulas are used only in Sections~\ref{sec:bg flip} and \ref{sec:strongly invertible}, where we compare the evaluations of a foam and its mirror or its orientation reversal.

A \textit{coloring} $c$ of a closed $\gl_N$ foam $F$ assigns an $\ell(f)$-subset $c(f)\subset\{1,\dots,N\}$ to each facet $f$ of $F$ subject to the following flow condition: if $f,g,h$ are three facets adjacent to a seam and $\ell(f)+\ell(g)=\ell(h)$, then $c(f)\cup c(g)=c(h)$. Given a closed $\gl_N$ foam $F$ and a coloring $c$ of $F$, Robert and Wagner define a homogeneous rational function in the variables $X_1,\dots,X_N$ of the form \[\langle F,c\rangle=(-1)^{s(F,c)}\frac{P(F,c)}{Q(F,c)},\]where $P(F,c)$ is a polynomial, $Q(F,c)$ is a rational function, and $s(F,c)$ is an integer, all explained below. The evaluation of $F$ is defined as \[\langle F\rangle\coloneqq\sum_c\langle F,c\rangle,\]where the sum is over all colorings of $F$. Although each $\langle F,c\rangle$ is only a rational function, Robert and Wagner prove that $\langle F\rangle$ is a symmetric polynomial \cite[Proposition 2.18]{robert2020closed}.

The polynomial $P(F,c)$ is defined as \[P(F,c)=\prod_{f\text{ facet}}P_f(X_{c(f)})\in \Z[X_1,\dots,X_N],\]where $P_f$ is the decoration on $f$, and $X_{c(f)}$ is the set of variables $\{X_i\mid i\in c(f)\}$; $P_f(X_{c(f)})$ is well-defined since $P_f$ is required to be symmetric.

To define $Q(F,c)$ and $s(F,c)$, we first introduce some notation. For $i\in\{1,\dots,N\}$, the \textit{monochrome surface} $F_i(c)$ is the union of all facets whose colorings contain $i$; it turns out to be a closed oriented surface. For $i,j\in\{1,\dots,N\}$ with $i<j$, the \textit{bichrome surface} $F_{ij}(c)$ is the union of all facets whose colorings contain exactly one of $i$ and $j$; it is also a closed oriented surface. In this situation, consider a seam with adjacent facets $f,g,h$ such that $i\in c(f)$, $j\in c(g)$, and $i,j\in c(h)$. The seam is said to be \textit{positive} with respect to $i<j$ if the cyclic order of the facets around the seam is $(f,g,h)$; otherwise, it is said to be \textit{negative}. The union of the seams that are positive (resp. negative) with respect to $i<j$ is a collection of simple closed curves in $F_{ij}(c)$, and we let $\theta_{ij}^+(F,c)$ (resp. $\theta_{ij}^-(F,c)$) be the number of such curves. Let $\theta_{ij}(F,c)\coloneqq\theta_{ij}^+(F,c)+\theta_{ij}^-(F,c)$. With this notation, we define \[Q(F,c)=\prod_{1\le i<j\le N}(X_i-X_j)^{\chi(F_{ij}(c))/2},\]and\[s(F,c)=\sum_{i=1}^N i\cdot\chi(F_i(c))/2+\sum_{1\le i<j\le N}\theta_{ij}^+(F,c),\]where $\chi(\cdot)$ denotes the Euler characteristic.

\subsubsection*{State spaces}For each planar $\gl_N$ web $\Gamma$, the theory assigns a $\Z$-graded $R$-module $\lrb{\Gamma}$, called the \textit{state space} (or the \textit{evaluation}) of $\Gamma$, via the universal construction \cite{blanchet1995topological}, as follows. Let $V_\Gamma$ be the free $R$-module generated by foams bounding $\Gamma$ in the lower half space, viewed as foams from the empty set to $\Gamma$, equipped with the symmetric $R$-bilinear form $\langle-,-\rangle$ obtained by evaluating closed foams: for $F,G\in V_\Gamma$, \[\langle F,G\rangle\coloneqq\langle F\circ \overline{G}\rangle,\]where $\overline{G}$ is the $\gl_N$ foam from $\Gamma$ to $\emptyset$ obtained by reflecting $G$ across the $xy$-plane. Define \[\lrb{\Gamma}\coloneqq V_\Gamma/\operatorname{rad}\langle-,-\rangle.\]Note that this construction is automatically functorial: a foam $F$ from $\Gamma_1$ to $\Gamma_2$ induces a map \[\lrb{F}\colon\lrb{\Gamma_1}\to\lrb{\Gamma_2},\]well-defined up to isotopy of $F$ relative to the boundary.

The construction above in fact carries a quantum grading. The graded rank of $\lrb{\Gamma}$ can be computed via MOY calculus \cite{murakami1998homily}. In our setting, the local relations in MOY calculus can be found in \cite[Figure 2]{wang2024n}.

\subsubsection*{Link homology}We now explain how the previous construction leads to a link homology theory. An oriented crossing admits two resolutions, obtained by replacing the crossing either by two parallel $1$-labeled edges or by a $2$-labeled thick edge. Depending on whether the crossing is positive or negative, one of the resolutions is said to be the $0$-resolution while the other is the $1$-resolution; see Figure~\ref{fig:resolution}. Two resolutions of a crossing are related by a $\gl_N$ foam, called the \textit{zip} or \textit{unzip} foam, as depicted in Figure~\ref{fig:zipunzip}.

\begin{figure}[hbtp]
    \centering
    \includegraphics[width=0.2\linewidth]{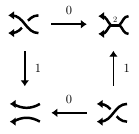}
    \caption{The $0$-resolution and the $1$-resolution of positive and negative crossings.}
    \label{fig:resolution}
\end{figure}

\begin{figure}[hbtp]
    \centering
    \includegraphics{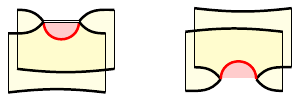}
    \caption{The zip foam (left) and the unzip foam (right).}
    \label{fig:zipunzip}
\end{figure}

Let $D$ be a diagram of an oriented link $K$ with $n$ (ordered) crossings, of which $n_+$ are positive and $n_-$ are negative. A \textit{complete resolution} of $D$ is a map $v\colon\{1,2,\dots,n\}\to\{0,1\}$, which gives rise to a (planar) $\gl_N$ web $D_v$ by assigning the $v(i)$-resolution to the $i$-th crossing. Let \[|v|\coloneqq\sum_{i=1}^n v(i)\]denote the $\ell^1$-norm of $v$. All such $v$ are encoded in the \textit{cube category} $\{0,1\}^n$, viewed as the poset of subsets of $\{1,2,\dots,n\}$ ordered by inclusion. An \textit{edge} in the cube category is a pair $u\to v$ such that $u(i)\le v(i)$ for all $i$ and $|v-u|=1$. The cube category naturally carries a grading by $|v|$.

All complete resolutions of $D$ can be organized into the \textit{cube of resolutions} of $D$: it is a functor from the cube category to the category of $\gl_N$ webs and foams, sending $v$ to $D_v$ and an edge $u\to v$ to the foam $F_{uv}$, the zip or unzip foam near the unique crossing $i$ such that $u(i)=0$ and $v(i)=1$. Applying the evaluation functor $\lrb{-}$ to the cube of resolutions and taking iterated mapping cones gives rise to a chain complex $\lrb{D}$ of $R$-modules. More precisely, all constructions above are in fact carried out in a graded setting; this grading is called the quantum grading, or $q$-grading. We introduce a homological grading, or $h$-grading, on $\lrb{D_v}$ by declaring that it lies in homological grading $0$. Now for integers $i,j$ and a bigraded $R$-module $V$, define $h^iq^jV$ to be the bigraded $R$-module obtained by shifting the bigrading by $(i,j)$. The Khovanov--Rozansky $\gl_N$ chain complex is then defined as\[\lrb{D}\coloneqq \left(h^{-n_+}q^{Nn_+-(N-1)n_-}\bigoplus_{v\in\{0,1\}^n}h^{|v|}q^{-|v|}\lrb{D_v},\,d=\sum_{u\to v}(-1)^{s(u,v)}\lrb{F_{uv}}\right),\]where $s(u,v)$ is the number of $1$'s in $u$ before the unique $1$ in $v-u$. The chain homotopy type of $\lrb{D}$ is a functorial invariant of the oriented link $K$ \cite{queffelec2016sln,robert2020closed,ehrig2018functoriality}, and the homology of $\lrb{D}$ is the Khovanov--Rozansky $\gl_N$ link homology of $K$.

For $v\in\{0,1\}^n$, let $n_-(v)$ be the number of negative crossings $i$ with $v(i)=1$. Let $o\in\{0,1\}^n$ be the vertex with $o(i)=1$ if the $i$-th crossing is positive and $o(i)=0$ otherwise. We call $o$, as well as the web $D_o$, the \textit{oriented resolution} of $D$. It replaces each crossing by two parallel $1$-labeled edges, and $\lrb{D_o}$ lies in homological degree $0$.

\subsubsection*{Reordering crossings}The complex $\lrb{D}$ depends a priori on the ordering of the crossings of $D$, and we fix the isomorphisms relating different orderings. For a permutation $\sigma$ of $\{1,\dots,n\}$, let $D^\sigma$ be the diagram $D$ with its $i$-th crossing relabeled as the $\sigma(i)$-th crossing. The permutation $\sigma$ acts on the cube category $\{0,1\}^n$ by $(\sigma v)(\sigma(i))=v(i)$, and we have $D^\sigma_{\sigma v}=D_v$. Let \[\operatorname{inv}_\sigma(v)\coloneqq\#\{(i,j)\mid i<j,\ \sigma(i)>\sigma(j),\ v(i)=v(j)=1\}.\]

    The \textit{reordering isomorphism} $P_\sigma\colon\lrb{D}\to\lrb{D^\sigma}$ is defined on $\lrb{D_v}$ as $(-1)^{\operatorname{inv}_\sigma(v)+\operatorname{inv}_\sigma(o)}$ times the identity map from $\lrb{D_v}$ to $\lrb{D^\sigma_{\sigma v}}$. It is easy to verify the following lemma.

\begin{lem}\label{lem:reordering}
    The map $P_\sigma$ is an isomorphism of chain complexes, and it acts as the identity on $\lrb{D_o}$. For permutations $\sigma$ and $\sigma'$, we have $P_{\sigma'}\circ P_\sigma=P_{\sigma'\sigma}$.
\end{lem}


The factor $(-1)^{\operatorname{inv}_\sigma(v)}$ is the usual sign that makes $P_\sigma$ a chain map, and the constant $(-1)^{\operatorname{inv}_\sigma(o)}$ normalizes $P_\sigma$ to be the identity on the oriented resolution. This agrees with the canopolis of \cite{ehrig2018functoriality}, where the complex of each crossing has its oriented resolution in homological degree $0$, so that reordering isomorphisms act as the identity on oriented resolutions; see \cite[Convention 3.1 and Remark 4.11]{ehrig2018functoriality}. We always identify complexes of the same diagram with different orderings of crossings via $P_\sigma$.

\subsubsection*{The canopolis of foams}Let $\FoamN$ be the canopolis of $\gl_N$ webs and foams from \cite[Section 2.3]{ehrig2018functoriality}, with integer coefficients \cite[Remark 2.17]{ehrig2018functoriality}. Its objects are webs, possibly with boundary points, its morphisms are linear combinations of foams in $D^2\times I$, whose side boundary is the product of the boundary points with $I$, and planar compositions glue webs and foams along boundary points. A linear combination $\sum_ja_jH_j$ of foams $H_j\colon\Gamma_0\to\Gamma_1$ vanishes in $\FoamN$ if and only if \[\sum_ja_j\langle H_j\cup G\rangle=0\] for every foam $G$ that can be glued to the $H_j$ along their common boundary, where $H_j\cup G$ denotes the resulting closed foam \cite[Section 2.3.3]{ehrig2018functoriality}. In particular, $\operatorname{Hom}_{\FoamN}(\emptyset,\Gamma)=\lrb{\Gamma}$ for a closed web $\Gamma$. For a tangle diagram $T$, the construction above gives webs $T_v$ with boundary and a chain complex $\lrb{T}$ over $\FoamN$, and we use the notation $n_-(v)$ and $o$ for $T$ as well. The complex of a link diagram is the planar composition of the complexes of its pieces up to sign.

\subsubsection*{Functoriality}

The complex $\lrb{\cdot}$ is functorial with respect to tangle cobordisms \cite{ehrig2018functoriality}. Recall that cobordism maps in $\gl_N$ link homology are defined by decomposing a link cobordism into a sequence of elementary moves, i.e., a movie. The following theorem is an integral version of \cite[Theorem 4.5]{ehrig2018functoriality}, cf. \cite[Lemma 3.22 and Remark 4.8]{ehrig2018functoriality}.

\begin{thm}[{\cite[Theorem 4.5]{ehrig2018functoriality}}]\label{thm:functoriality}
    Let $D_0$ and $D_1$ be diagrams of two oriented links. If two movies from $D_0$ to $D_1$ represent link cobordisms that are isotopic relative to the boundary, then the chain maps $\lrb{D_0}\to\lrb{D_1}$ induced by the two movies are chain homotopic.
\end{thm}

We also need the following integral version of \cite[Lemma 4.6]{ehrig2018functoriality}, cf. \cite[Remarks 2.17 and 4.8]{ehrig2018functoriality}, which also serves as one of the key ingredients in the proof of the functoriality in \cite{ehrig2018functoriality}.

\begin{lem}\label{lem:simple tangle}
    Let $T$ be a tangle diagram that is isotopic to a trivial tangle without fixing the boundary points. Then the chain homotopy classes of chain maps $\lrb{T}\to\lrb{T}$ of bidegree $(0,0)$ form the group $\Z\cdot\idd$. Consequently, if $f,g\colon\lrb{T}\to\lrb{T'}$ are homotopy equivalences of bidegree $(0,0)$ for some tangle diagram $T'$, then $f\simeq\pm g$.
\end{lem}

\subsection{MOY foams and Soergel calculus}\label{sec:bg MOY foams}

Most of our computation takes place on a restricted class of webs and foams. In this subsection, we introduce them and set up the diagrammatic language used throughout the paper.

\subsubsection*{Coherently oriented webs}A \textit{coherently oriented braid} is a braid whose strands are oriented in the same direction, and a \textit{coherently oriented braid closure} is the link diagram obtained by closing a coherently oriented braid. We say a $\gl_N$ web is \textit{coherently oriented} if it can be embedded into an annulus with core $S^1$ such that all edges are directed in the sense that the tangent vector of the edge has a positive inner product with the tangent vector of $S^1$ at each point. We say a $\gl_N$ web is \textit{elementary} if it contains only $1$- and $2$-labeled edges and no $2$-labeled circles. It is easy to see that the $\gl_N$ webs arising in the cube of resolutions of an uncolored coherently oriented braid closure are all elementary and coherently oriented. We will mostly work with elementary and coherently oriented webs by restricting our link diagrams to coherently oriented braid closures. 

\subsubsection*{MOY foams}

The state spaces of elementary coherently oriented webs can be described using a much smaller class of foams, which we call \textit{MOY foams}. Recall that a general $\gl_N$ foam is assembled from a large collection of local pieces near its singular points; see \cite[Remark 2.11]{ehrig2018functoriality}. A \textit{MOY foam} is a foam that can be written as a composition of ordinary foams with labels $1$ and $2$, possibly carrying dots, together with the following four singular local pieces: \begin{itemize}
    \item Zip and unzip foams with $1$ and $2$ labels, as depicted in Figure~\ref{fig:zipunzip}.
    \item Digon creation and annihilation foams with $1$ and $2$ labels, as depicted in Figure~\ref{fig:digon}.
    \item The foam depicted in Figure~\ref{fig:starfoam}, which contains six singular vertices and a $3$-labeled facet.
    \item The foam depicted in Figure~\ref{fig:fourvalentfoam}, which is isotopic to the identity foam but switches the order of thick edges.
\end{itemize}

\begin{figure}[hbtp]
    \centering
    \includegraphics{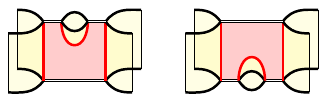}
    \caption{The digon creation foam (left) and the digon annihilation foam (right).}
    \label{fig:digon}
\end{figure}

\begin{figure}
    \centering
    \includegraphics[width=0.2\linewidth]{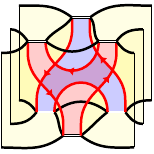}
    \caption{The foam corresponding to a six-valent vertex in Soergel calculus.}
    \label{fig:starfoam}
\end{figure}

\begin{figure}[hbtp]
    \centering
    \includegraphics[scale=0.8]{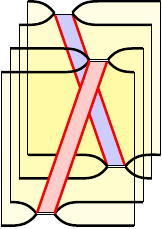}
    \caption{The foam corresponding to a four-valent vertex in Soergel calculus.}
    \label{fig:fourvalentfoam}
\end{figure}

\begin{lem}\label{lem:MOY foam generates eco webs}
    Let $\Gamma$ be an elementary coherently oriented web. Then the state space $\lrb{\Gamma}$ is generated by images of MOY foams between unlinks and $\Gamma$.
\end{lem}

\begin{proof}
    This is a consequence of the categorified MOY calculus \cite{mackaay2009sl}. More precisely, \cite[Lemma 9.2]{mackaay2009sl} realizes the isomorphisms between state spaces related by local relations in MOY calculus via foams. Rasmussen \cite[Corollary 4.2]{Rasmussen2006SomeDO} proves that elementary coherently oriented $\gl_N$ webs can be evaluated via MOY calculus without using the square relation described in \cite[Lemma 9.2(3)]{mackaay2009sl}. All other foams in \cite[Lemma 9.2]{mackaay2009sl} are compositions of the local pieces we allow for MOY foams. For instance, the disoriented digon removal relation in \cite[Lemma 9.2(2)]{mackaay2009sl} is realized by the composition of a zip foam and a birth foam (that creates a $1$-labeled circle).
\end{proof}

\subsubsection*{Soergel calculus}

Soergel calculus of type $A_{n-1}$ is a graphical
description of the $2$-category of Soergel bimodules, which categorifies the Hecke algebra for the symmetric group $S_n$. For any
$N\ge 2$, this $2$-category admits a $2$-functor to a certain monoidal subcategory of $\gl_N$ webs and foams with $2n$ boundary
points with suitable orientations; see, e.g., \cite{mackaay2010diagrammatic}. For us, Soergel calculus will only be used as a shorthand notation for foam calculations, following \cite{morrison2022invariants,stroppel2024braiding}.

In the $A_{n-1}$ calculus, we define the strictly pivotal and strictly monoidal $\Bbbk$-linear category $\mathcal{D}_n^{\textrm{free}}$, freely generated by $n-1$ Frobenius algebra objects $(c_i,m_i,\Delta_i,\varepsilon_i,\eta_i),\, i=1,\dots,n-1$, monoidal product $\circ_1$, and additional generating morphisms \begin{align*}
c_i\circ_1 c_j\circ_1 c_i&\to c_j\circ_1 c_i\circ_1 c_j,\, |i-j|=1,\\
c_i\circ_1c_k&\to c_k\circ_1 c_i,\, |i-k|>1,\\
x_k\colon\mathbf{1}&\to\mathbf{1},\,k=1,\dots,n.
\end{align*}These data are graded; see \cite[Definition 4.1]{stroppel2024braiding}.

We now present a diagrammatic description of the category $\mathcal{D}_n^{\textrm{free}}$. We denote the generating objects by colors $\{{\color{red}\textrm{red}},\,{\color{blue}\textrm{blue}},\,{\color{green}\textrm{green}},\cdots\}$; the generating morphisms are depicted as follows: \[\includegraphics{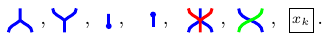}\]Here the first six diagrams are called \textit{merge}, \textit{split}, \textit{start dot}, \textit{end dot}, \textit{six-valent vertex}, and \textit{four-valent vertex}, respectively. The diagrams are read from bottom to top, the monoidal structure $\circ_1$ is displayed by horizontal juxtaposition, and the composition of morphisms is displayed by vertically stacking planar diagrams. To obtain the diagrammatic description of Soergel bimodules of type $\mathfrak{sl}_n$, we further impose relations on morphisms; see \cite[Definition 4.4, Remark 4.7]{stroppel2024braiding} for a complete list.

We now interpret webs and foams via diagrams in Soergel calculus. In the $A_{n-1}$ calculus, an object corresponds to an $n$-strand web as follows. The generating object $c_i$ corresponds to the $\gl_N$ web given by $n$ parallel coherently oriented strands with the $i$-th and $(i+1)$-th strands replaced by the web with a thick edge (usually displayed in the $i$-th color), and tensor product of generating objects corresponds to horizontal juxtaposition: \[\includegraphics{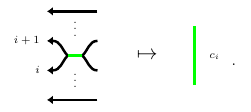}\]

MOY foams are exactly those foams that can be represented diagrammatically in Soergel calculus, which we now explain. Start dots and end dots correspond to zip and unzip foams, respectively. Merge and split correspond to digon creation and annihilation foams, respectively. The six-valent vertex corresponds to the foam depicted in Figure~\ref{fig:starfoam}. The four-valent vertex corresponds to the foam depicted in Figure~\ref{fig:fourvalentfoam}. The $x_k$ morphism corresponds to a dot on the $k$-th strand. 
\subsubsection*{Cobordism maps via Soergel calculus}
Soergel calculus provides a convenient language for describing cobordism maps. For our purposes, we record some maps induced by Reidemeister moves, cf. \cite[Sections 3.4 and 3.5]{morrison2022invariants}. We also record the atomic slide map introduced in \cite[Proposition 6.3]{stroppel2024braiding}.
\begin{prop}\label{prop:RIImove}
    The maps induced by Reidemeister II moves are given as follows.
    \[\RIImove\]
\end{prop}

We also use Proposition~\ref{prop:RIImove} to illustrate our notation. The upper left tangle has two crossings: the left one is negative and the right one is positive. The upper right diagram is the cube of resolutions of the tangle, depicted using Soergel calculus. The leftmost diagram in the upper right resolves both crossings to $0$-resolutions, and the diagram reads as the same order as the crossings: the left blue dot indicates the trivial (parallel) resolution of the left crossing, and the right blue object indicates the nontrivial resolution of the right crossing. The differentials are zip and unzip foams, which are indicated by start and end dots in the Soergel calculus. 

\begin{prop}\label{prop:RIIImove}
    The maps induced by Reidemeister III moves are given as follows.

    \[\RIIImovea\]

    \[\RIIImoveb\]
\end{prop}

\begin{prop}[{\cite[Proposition 6.3]{stroppel2024braiding}}]
    \label{prop:atomicslide}
    There are chain maps

    \[\slide_{\one_1,B_1}:= \hspace{-.5cm}\FigSlide,\quad \slide_{B_1,\one_1}:=\hspace{-.5cm} \FigSlidee \]

    which are invertible up to homotopy. The inverses are given by the chain maps
    \[\slide^{-1}_{\one_1,B_1}:= \hspace{-.5cm}\FigSlidei, \quad
        \slide^{-1}_{B_1,\one_1}:= \hspace{-.5cm}\FigSlideei.\]
\end{prop}

\section{The flip symmetry}\label{sec:bg flip}

In \cite{chen2025flip}, Chen and the author define the flip map on the Khovanov chain complex of a link diagram, and prove that it is a chain map and a link invariant up to homotopy. A large part of this construction is formal, and we generalize it to $\gl_N$ homology in this section.

Let $D$ be a diagram of an oriented link $K$ in $\RR^3$. We think of $D$ as lying on the $xy$-plane. Its \textit{flip diagram} $D^*$ is another diagram of $K$ obtained by rotating $K$ in $\RR^3$ by $\pi$ around an axis in the $xy$-plane and projecting it down to the same plane. From a planar perspective, $D^*$ is obtained by taking the mirror of $D$ and switching all the crossings. The two diagrams $D$ and $D^*$ are related by the cobordism in $I\times\RR^3$ obtained as the trace of the rotation, which we call the \textit{flip cobordism}. We always assume that the axis of rotation is the $x$-axis and that the rotation is clockwise when viewed from the positive $x$-axis; choosing other conventions will yield the same map up to homotopy \cite[Remark 3.2]{chen2025flip}. 

We now construct the maps \[\rho,\,\eta\colon\lrb{D}\to\lrb{D^*}\]mentioned in the introduction. The map $\rho$ is the chain homotopy equivalence induced by the flip cobordism. It is well-defined up to chain homotopy by Theorem~\ref{thm:functoriality}.

Defining the map $\eta$ requires slightly more work. For a planar $\gl_N$ web $\Gamma$, let $\Gamma^*$ be the planar $\gl_N$ web obtained by rotating $\Gamma$ by $\pi$ around the $x$-axis; it is the same as the mirror of $\Gamma$ since $\Gamma$ is planar. By the universal construction, $\lrb{\Gamma}$ is generated by foams in the lower half space $\{z\le 0\}$ with boundary $\Gamma$. Reflecting such foams across the $xz$-plane yields generators of $\lrb{\Gamma^*}$. For a foam $F$, we denote its reflection by $F^*$. The reflection naturally identifies the generators, but it does not directly descend to a map between state spaces over $\Z$ in general. To obtain a well-defined map \[\eta_\Gamma\colon\lrb{\Gamma}\to\lrb{\Gamma^*},\]we have to correct the signs by the following weighted Euler characteristics.

\begin{defn}\label{defn:weighted euler}
    Let $\Sigma$ be a $\gl_N$ web or foam, possibly with boundary. Choose a finite cell decomposition of $\Sigma$ such that each open cell is contained in a single stratum of $\Sigma$, or a stratum of the boundary of $\Sigma$. For a cell $\sigma$, let $m(\sigma)$ be the maximal label of the facets (resp. edges) of $\Sigma$ whose closures contain $\sigma$. The \textit{weighted Euler characteristic} of $\Sigma$ is defined as \[\omega(\Sigma)\coloneqq\sum_{\sigma}(-1)^{\dim\sigma}\binom{m(\sigma)}{2}.\]
\end{defn}

\begin{lem}\label{lem:omega properties}
    The weighted Euler characteristic $\omega(\Sigma)$ does not depend on the choice of the cell decomposition. Moreover, it has the following properties.
    \begin{enumerate}[(1)]
        \item For a web or foam $\Sigma$, we have $\omega(\Sigma^*)=\omega(\Sigma)$. For a web $\Gamma$, we have $\omega(\Gamma\times I)=\omega(\Gamma)$.
        \item If $\Sigma$ is obtained by gluing $\Sigma_1$ and $\Sigma_2$ along a common part $Y$ of their boundaries, then \[\omega(\Sigma)=\omega(\Sigma_1)+\omega(\Sigma_2)-\omega(Y).\] In particular, for foams $H_1\colon\Gamma_0\to\Gamma_1$ and $H_2\colon\Gamma_1\to\Gamma_2$, we have \[\omega(H_2\circ H_1)=\omega(H_2)+\omega(H_1)-\omega(\Gamma_1),\]and $\omega$ is additive under planar compositions of webs and foams along $1$-labeled boundary points.
        \item If all labels of $\Sigma$ are at most $3$, then $\omega(\Sigma)\equiv\chi(\Sigma^{(2)})\pmod 2$, where $\Sigma^{(2)}$ is the union of the closures of all facets (resp. edges) of $\Sigma$ with labels $2$ and $3$.
    \end{enumerate}
\end{lem}

\begin{proof}

    For $j\ge 1$, let $\Sigma_{\ge j}$ be the union of the closures of all facets (resp. edges) of $\Sigma$ with labels at least $j$. Since each open cell lies in a single stratum, $\Sigma_{\ge j}$ is a subcomplex, and a cell $\sigma$ lies in $\Sigma_{\ge j}$ if and only if $m(\sigma)\ge j$. We have \begin{equation}\label{eqn:weighted Euler char}
    \omega(\Sigma)=\sum_{\sigma}(-1)^{\dim\sigma}\sum_{j=1}^{m(\sigma)}(j-1)=\sum_{j\ge 1}(j-1)\,\chi(\Sigma_{\ge j}),
    \end{equation}which does not depend on the cell decomposition.

    Now (1) and (2) easily follow from \eqref{eqn:weighted Euler char} and the additivity of the usual Euler characteristic. For (3), we have $\Sigma_{\ge 2}=\Sigma^{(2)}$ and $\Sigma_{\ge a}=\emptyset$ for $a\ge 4$, so \[\omega(\Sigma)=\chi(\Sigma^{(2)})+2\chi(\Sigma_{\ge 3})\equiv\chi(\Sigma^{(2)})\pmod 2.\]
\end{proof}

\begin{lem}\label{lem:reflecting closed foams}
    Let $F$ be a closed $\gl_N$ foam, and let $F^*$ be the closed $\gl_N$ foam obtained by reflecting $F$ across a plane in $\RR^3$. Then $\langle F^*\rangle=(-1)^{\omega(F)}\langle F\rangle$.
\end{lem}

\begin{proof}
    Recall from Section~\ref{sec:bg gln homology} that the evaluation of a closed $\gl_N$ foam $F$ takes the form \[\langle F\rangle=\sum_c\langle F,c\rangle,\]where the sum is over all colorings of $F$, and \[\langle F,c\rangle=(-1)^{s(F,c)}\frac{P(F,c)}{Q(F,c)}.\]
    
    The reflection induces an obvious bijection between the colorings of $F$ and $F^*$. For a coloring $c$ of $F$, let $c^*$ be the corresponding coloring of $F^*$. It suffices to prove that \[\langle F^*,c^*\rangle=(-1)^{\omega(F)}\langle F,c\rangle\]for every coloring $c$. The polynomial $P(F,c)$ is defined as a product of the decorations on facets, which are symmetric polynomials. Since the reflection preserves the decorations, we have \[P(F,c)=P(F^*,c^*).\]The rational function $Q(F,c)$ is defined as a product of $(X_i-X_j)^{\chi(F_{ij}(c))/2}$. The reflection preserves the bichrome surfaces $F_{ij}(c)$, so we have\[Q(F^*,c^*)=Q(F,c).\]
    
    It remains to compare the signs. The reflection also preserves the monochrome surfaces, so the first term in $s(F,c)$ is preserved under reflection. Reflection changes cyclic orders and identifies positive curves in $F^*$ with negative curves in $F$. Therefore, we have 
    \begin{align*}
        s(F,c)-s(F^*,c^*)&=\sum_{1\le i<j\le N}\left(\theta_{ij}^+(F,c)-\theta_{ij}^-(F,c)\right)\\ &\equiv \sum_{1\le i<j\le N}\theta_{ij}(F,c)\pmod 2\\ &\equiv \sum_{1\le i<j\le N}\chi(F_{i\cap j}(c))\pmod 2.
    \end{align*}Here $F_{i\cap j}(c)$ denotes the union of all facets whose colorings contain both $i$ and $j$, and the last line follows from \cite[Lemma 2.7]{robert2020closed}. The latter sum only depends on the labels of the facets: we choose a cellular structure on $F$ compatible with the stratification as in Definition~\ref{defn:weighted euler}, and 
    \begin{align*}
        \sum_{1\le i<j\le N}\chi(F_{i\cap j}(c))&=\sum_{1\le i<j\le N}\left(\sum_{f}\mathbf{1}_{i,j\in c(f)}-\sum_{e}\mathbf{1}_{i,j\in c(e)}+\sum_{v}\mathbf{1}_{i,j\in c(v)}\right)\\
        &=\sum_{f}\left(\sum_{1\le i<j\le N}\mathbf{1}_{i,j\in c(f)}\right)-\sum_{e}\left(\sum_{1\le i<j\le N}\mathbf{1}_{i,j\in c(e)}\right)+\sum_{v}\left(\sum_{1\le i<j\le N}\mathbf{1}_{i,j\in c(v)}\right)\\
        &=\sum_{f}\binom{m(f)}{2}-\sum_{e}\binom{m(e)}{2}+\sum_{v}\binom{m(v)}{2}=\omega(F).
    \end{align*}
      Here $f$, $e$, $v$ range over the cells of dimension $2,1,0$. For a cell $\sigma$, $c(\sigma)$ denotes the coloring of a facet of label $m(\sigma)$ whose closure contains $\sigma$. By the flow condition, $c(\sigma)$ contains the colorings of all facets whose closures contain $\sigma$, so $\sigma$ lies in $F_{i\cap j}(c)$ if and only if $i,j\in c(\sigma)$.
\end{proof}

To define $\eta_\Gamma$ and to compare $\eta$ with maps induced by Reidemeister moves, we define a signed reflection in $\FoamN$. For a foam $H\colon\Gamma_0\to\Gamma_1$ between webs that may have boundary, we define \[\delta(H)\coloneqq\omega(H)-\omega(\Gamma_0),\ H^\dagger\coloneqq(-1)^{\delta(H)}H^*.\]We call $H^\dagger$ the \textit{signed reflection} of $H$.

\begin{prop}\label{prop:signed reflection}
    Signed reflection extends linearly to a well-defined functor from $\FoamN$ to itself, which sends a web $\Gamma$ to $\Gamma^*$ and a foam $H$ to $H^\dagger$. It preserves the quantum grading and satisfies $(H^\dagger)^\dagger=H$. Moreover, it sends a planar composition of foams along $1$-labeled boundary points to the reflected planar composition of their signed reflections.
\end{prop}

\begin{proof}
    By Lemma~\ref{lem:omega properties}, we have $\delta(H_2\circ H_1)=\delta(H_2)+\delta(H_1)$ and $\delta(\idd_\Gamma)=0$, so signed reflection preserves compositions and identities of foams. For the same reason, $\delta$ is additive under planar compositions along $1$-labeled boundary points. Reflection preserves $\omega$ and the quantum degree of foams, so $\delta(H^*)=\delta(H)$, $(H^\dagger)^\dagger=H$, and signed reflection preserves the quantum grading.

    It remains to show that signed reflection preserves the relations in $\FoamN$. Let \[x=\sum_ja_jH_j\] be a linear combination of foams $H_j\colon\Gamma_0\to\Gamma_1$, which share the same boundary $\partial H_j$. Recall that $x=0$ in $\FoamN$ if and only if \[\sum_ja_j\langle H_j\cup G\rangle=0\] for every foam $G$ that can be glued to the $H_j$ along $\partial H_j$. By Lemma~\ref{lem:omega properties}, we have \[\omega(H_j\cup G)=\omega(H_j)+\omega(G)-\omega(\partial H_j).\] Combining with Lemma~\ref{lem:reflecting closed foams}, we have \[\langle H_j^\dagger\cup G^*\rangle=(-1)^{\delta(H_j)}\langle(H_j\cup G)^*\rangle=(-1)^{\omega(G)-\omega(\partial H_j)-\omega(\Gamma_0)}\langle H_j\cup G\rangle.\]The sign on the right does not depend on $j$. Since every foam that can be glued to the $H_j^*$ along $\partial H_j^*$ is of the form $G^*$, $x=0$ implies \[\sum_j a_j H_j^\dagger=0.\]
\end{proof}

For a closed web $\Gamma$ and a foam $F\colon\emptyset\to\Gamma$, we have $\delta(F)=\omega(F)$. By Proposition~\ref{prop:signed reflection}, signed reflection restricts to an isomorphism of graded $R$-modules \[\eta_\Gamma\colon\lrb{\Gamma}\to\lrb{\Gamma^*},\qquad[F]\mapsto(-1)^{\omega(F)}[F^*],\]with inverse $\eta_{\Gamma^*}$. Moreover, for a foam $H\colon\Gamma_0\to\Gamma_1$ between closed webs, functoriality of signed reflection gives
\begin{equation}\label{eqn:eta natural}
    \eta_{\Gamma_1}\circ\lrb{H}=(-1)^{\delta(H)}\lrb{H^*}\circ\eta_{\Gamma_0}.
\end{equation}


There is an obvious identification between the crossings in $D$ and those in $D^*$ that preserves the positivity of crossings, so $n_\pm(D)=n_\pm(D^*)$. We define $\eta$ as a direct sum \[\eta\coloneqq h^{-n_+}q^{Nn_+-(N-1)n_-}\bigoplus_{v\in \{0,1\}^n}h^{|v|}q^{-|v|}(-1)^{n_-(v)}\eta_{D_v}\colon\lrb{D}\to\lrb{D^*}.\]When we need to specify the diagram, we write $\eta_D$ for $\eta$. By construction, $\eta_D$ preserves the cubical grading, and $\eta_{D^*}\circ\eta_D=\idd$. 

\begin{lem}\label{lem:eta chain map}
    The map $\eta_D\colon\lrb{D}\to\lrb{D^*}$ is a chain map.
\end{lem}

\begin{proof}
Let $u\to v$ be an edge in the cube, and let $i$ be the crossing with $u(i)=0$ and $v(i)=1$. The foam $F_{uv}^*$ is a zip or unzip foam of $D^*$ at the $i$-th crossing, and the signs $(-1)^{s(u,v)}$ in the differentials of $\lrb{D}$ and $\lrb{D^*}$ agree. By (\ref{eqn:eta natural}), we have \[\eta_{D_v}\circ\lrb{F_{uv}}=(-1)^{\delta(F_{uv})}\lrb{F_{uv}^*}\circ\eta_{D_u}.\]It suffices to prove that \[\delta(F_{uv})\equiv n_-(v)-n_-(u)\pmod 2.\]Note that $D_u^{(2)}$ is a disjoint union of closed intervals, one for each thick edge of $D_u$. If the $i$-th crossing is negative, then $F_{uv}$ is a zip foam, and $F_{uv}^{(2)}$ is the disjoint union of $D_u^{(2)}\times I$ and a closed disk, so $\delta(F_{uv})$ is odd by Lemma~\ref{lem:omega properties}(3). If the $i$-th crossing is positive, then $F_{uv}$ is an unzip foam, and $F_{uv}^{(2)}$ is obtained from $D_u^{(2)}\times I$ by replacing $e\times I$ with a closed disk, where $e$ is the thick edge at the $i$-th crossing, so $\delta(F_{uv})$ is even. In both cases, we have $\delta(F_{uv})\equiv n_-(v)-n_-(u)\pmod 2$, so $\eta_D$ is indeed a chain map.
\end{proof}

\begin{defn}
    Let $D$ be a diagram of an oriented link $K$. The \textit{flip map} is defined as the composition \[\kappa\coloneqq\eta^{-1}\circ\rho\colon\lrb{D}\xrightarrow{\rho}\lrb{D^*}\xrightarrow{\eta^{-1}}\lrb{D}.\]
\end{defn}

\begin{rem}
    Our definition of the flip map here is slightly different from that in \cite{chen2025flip}, where the flip map is defined as the composition $\rho^{-1}\circ\eta$. The two definitions are equivalent up to homotopy since $\kappa$ is a homotopy involution, cf. \cite[Lemma 3.5]{chen2025flip}.
\end{rem}

To prove that the flip map gives a link invariant in an appropriate sense, we first prove that $\eta$ is compatible with Reidemeister moves. For a tangle diagram $T$, let $T^*$ be the flip tangle of $T$, obtained by taking the mirror of $T$ and switching all crossings. Reflecting all webs in the complex $\lrb{T}$ and replacing each foam by its signed reflection gives a complex $\lrb{T}^\dagger$, whose vertices are $(T_v)^*=(T^*)_v$. The computation above shows that \[\alpha_T\coloneqq\bigoplus_{v}(-1)^{n_-(v)}\idd\colon\lrb{T}^\dagger\to\lrb{T^*}\]is an isomorphism of complexes. For a chain map $g\colon\lrb{T}\to\lrb{T'}$, where $T$ and $T'$ have the same boundary, we define \[\widehat{g}\coloneqq \alpha_{T'}\circ g^\dagger\circ \alpha_T^{-1}\colon\lrb{T^*}\to\lrb{T'^*},\]where $g^\dagger$ is obtained from $g$ by replacing each foam by its signed reflection. By Proposition~\ref{prop:signed reflection}, the assignment $g\mapsto\widehat{g}$ preserves compositions, identities and chain homotopies, and it is compatible with planar compositions. In particular, for link diagrams $D_1$ and $D_2$, (\ref{eqn:eta natural}) gives
\begin{equation}\label{eqn:hat and eta}
    \widehat{g}\circ\eta_{D_1}=\eta_{D_2}\circ g.
\end{equation}

\begin{prop}\label{prop:eta reidemeister}
    Let $D_1$ and $D_2$ be link diagrams related by a Reidemeister move, and let $f\colon\lrb{D_1}\to\lrb{D_2}$ be the induced chain map. Let $f^*\colon\lrb{D_1^*}\to\lrb{D_2^*}$ be the chain map induced by the corresponding move between the flip diagrams. Then the following diagram commutes up to homotopy:
    \[
        \tikzexternaldisable
            \begin{tikzcd}
                \lrb{D_1} \arrow[r, "\eta_{D_1}"] \arrow[d, "f"] & \lrb{D_1^*} \arrow[d, "f^*"] \\
                \lrb{D_2} \arrow[r, "\eta_{D_2}"]                & \lrb{D_2^*}
            \end{tikzcd}.
        \tikzexternalenable
    \]
\end{prop}

\begin{proof}
    It suffices to prove the statement for the atomic tangle diagrams $T$ and $T'$ related by a Reidemeister move. Let $r$ be the chain map induced by the move, defined in \cite[Section 3.3]{ehrig2018functoriality}, and let $r^*$ be the chain map induced by the reflected move. By (\ref{eqn:hat and eta}), it suffices to prove that the following diagram of tangle complexes commutes up to homotopy:
    \[
        \tikzexternaldisable
            \begin{tikzcd}
                \lrb{T}^\dagger \arrow[r, "\alpha_T"] \arrow[d, "r^\dagger"] & \lrb{T^*} \arrow[d, "r^*"] \\
                \lrb{T'}^\dagger \arrow[r, "\alpha_{T'}"]                & \lrb{T'^*}
            \end{tikzcd}.
        \tikzexternalenable
    \]
    In other words, $\widehat{r}=\alpha_{T'}\circ r^\dagger\circ\alpha_T^{-1}$ is chain homotopic to $r^*$. Both $\widehat{r}$ and $r^*$ are chain homotopy equivalences of bidegree $(0,0)$, so we have $\widehat{r}\simeq\epsilon r^*$ for a fixed $\epsilon\in\{\pm1\}$ by Lemma~\ref{lem:simple tangle}, as in \cite[Section 4.2]{ehrig2018functoriality}. To check $\epsilon=1$, we close up $T$ to a diagram of an unlink and compare the isomorphisms induced on homology on an explicit generator. Alternatively, one may proceed as in \cite[Section 4.3]{ehrig2018functoriality} to compare the signs on simple resolutions after specializing the generic parameters $X_1,\dots,X_N$ to distinct complex numbers. 
\end{proof}

We are now ready to prove that the flip map is a link invariant up to homotopy.

\begin{prop}\label{prop:flip invariance}
    Let $D_1$ and $D_2$ be two diagrams of the same oriented link $K$. Then the following diagram commutes up to homotopy:
    \[
        \tikzexternaldisable
            \begin{tikzcd}
                \lrb{D_1} \arrow[r, "\kappa_1"] \arrow[d, "\varphi"] & \lrb{D_1} \arrow[d, "\varphi"] \\
                \lrb{D_2} \arrow[r, "\kappa_2"]                      & \lrb{D_2}
            \end{tikzcd}.
        \tikzexternalenable
    \]
    Here $\kappa_1$ and $\kappa_2$ are the flip maps on $\lrb{D_1}$ and $\lrb{D_2}$, respectively, and the vertical map $\varphi$ is a fixed chain homotopy equivalence induced by a sequence of Reidemeister moves.
\end{prop}

\begin{proof}
    This follows from the same argument as in \cite[Theorem 3.8]{chen2025flip} by Theorem~\ref{thm:functoriality} and  Proposition~\ref{prop:eta reidemeister}.
\end{proof}

\begin{exmp}\label{ex:flip unknot}
    Let $U$ be the crossingless diagram of the unknot. The chain complex $\lrb{U}$ is a free $R$-module of rank $N$, concentrated in homological grading $0$, generated by $1,X,\dots,X^{N-1}$, where $X^k$ denotes the image of $1$ under the birth foam $\emptyset\to U$ with $k$ dots. The flip diagram $U^*$ is also crossingless, and we denote the generators of $\lrb{U^*}$ by $1,X^*,\dots,(X^*)^{N-1}$. By functoriality, \[\rho(X^k)=(X^*)^k=\eta(X^k)\]since there is no room for a homotopy. Therefore, the flip map $\kappa$ on $\lrb{U}$ is the identity map. More generally, for a crossingless diagram $U_l$ of an unlink with $l$ components, possibly nested, $\lrb{U_l}$ is concentrated in homological grading $0$ and generated by the classes of disjoint unions $F_0\colon\emptyset\to U_l$ of disks with dots, and the same argument gives $\rho([F_0])=[F_0^*]=\eta([F_0])$.
\end{exmp}

\section{The fork twist}\label{sec:flip 2 strands}

The filtration argument in \cite{chen2025flip} relies on the cup-sliding trick from \cite{rozansky2010categorification}, which becomes invalid when thick edges are present. In this section, we define and study maps induced by (a balanced version of) fork twists, cf. \cite[Equation (4.16)]{queffelec2016sln}. We then utilize these maps to study the minimal case where thick edges appear: the flip map on the $2$-strand braid with a single crossing. This will serve as the model case for our argument in Section~\ref{sec: flip n strand}.

\begin{defn}\label{defn:fork twist}
    The map induced by a positive fork twist is defined as follows: \[\tw\coloneqq\forktwist.\]Its inverse is defined as \[\tw^{-1}\coloneqq\forktwisti.\]
    Here we use the shorthand notation \begin{equation}\label{eqn:defn_black_dot}
        \begin{tikzpicture}[anchorbase,scale=.2]
            \draw[bl] (1,2.01) \dr (1.5,1.3) \ru (2,2.01) (1.5,-.01)\pu (1.5,1.3);
            \fill (1.1,1) circle (1.5mm);
        \end{tikzpicture}\coloneqq \begin{tikzpicture}[anchorbase,scale=.2]
            \draw[bl] (1,2.01) \dr (1.5,1.5) \ru (2,2.01) (1.5,-.01)\pu (1.5,1.01);
            \fill[bl] (1.5,1) circle (2.5mm);
        \end{tikzpicture}-\begin{tikzpicture}[anchorbase,scale=.2]
            \draw[bl] (1,2.01) \pd (1,1-.01)  (2,2.01)\pd (2,-.01);
            \fill[bl] (1,1) circle (2.5mm);
        \end{tikzpicture}.
    \end{equation}We require this to be rotationally invariant. Rotating (\ref{eqn:defn_black_dot}) by $2\pi/3$ yields \begin{equation}
        \begin{tikzpicture}[anchorbase,scale=.2]
            \draw[bl] (1,2.01) \dr (1.5,1.3) \ru (2,2.01) (1.5,-.01)\pu (1.5,1.3);
            \fill (1.9,1) circle (1.5mm);
        \end{tikzpicture}=-\begin{tikzpicture}[anchorbase,scale=.2]
            \draw[bl] (1,2.01) \dr (1.5,1.5) \ru (2,2.01) (1.5,-.01)\pu (1.5,1.01);
            \fill[bl] (1.5,1) circle (2.5mm);
        \end{tikzpicture}+\begin{tikzpicture}[anchorbase,scale=.2]
            \draw[bl] (2,2.01) \pd (2,1-.01)  (1,2.01)\pd (1,-.01);
            \fill[bl] (2,1) circle (2.5mm);
        \end{tikzpicture}.
    \end{equation}
\end{defn}

\begin{rems}
    The maps defined in Definition~\ref{defn:fork twist} differ from those in Proposition~\ref{prop:atomicslide} in that the latter are filtered pieces of maps induced by certain Reidemeister III moves, whereas the former do not arise from any Reidemeister move on tangles. Maps induced by negative fork twists can be defined similarly; we do not pursue them here, as they are not needed in this paper.
\end{rems}

\begin{prop}
    The maps $\tw$ and $\tw^{-1}$ are chain maps, and they are mutually homotopy inverse.
\end{prop}

\begin{proof}
    It is easy to check that $\tw$ and $\tw^{-1}$ are chain maps. It remains to prove they are mutually homotopy inverse. We prove $\tw^{-1}\circ\tw$ is chain homotopic to $\idd$:\[\tw^{-1}\circ\tw-\idd=\left(\twistdifference\right)=\left[d,\left(\twisthomotopy\right)\right].\]The other direction is similar.
\end{proof}

Let $M$ and $M'$ be chain complexes in an additive category. Recall that a \textit{very strong deformation retract} from $M$ to $M'$ \cite[Definition 2.9]{willis2021khovanov} is a chain map $\pi\colon M\to M'$ satisfying the following conditions:\begin{itemize}
    \item There is a chain map $\iota\colon M'\to M$ and a map $H\colon M\to M[-1]$ such that $\pi\circ\iota=\idd_{M'}$ and $\iota\circ\pi-\idd_M=dH-Hd$. In other words, $\pi$ is a strong deformation retract.

    \item The homotopy $H$ satisfies the side conditions: $H\circ H=0$, $\pi\circ H=0$, and $H\circ\iota=0$.
\end{itemize}

\begin{prop}\label{prop:2strands vsdr}
    Let $\rho_0$ be the map induced by the composition of a fork twist and a Reidemeister II move that eliminates crossings: \[\rho_0\colon\lrb{
            \begin{tikzpicture}[anchorbase,scale=.3,rotate=90]
                \BSbl{0}{0}
                \idwebl{1}{1.5}{1}
                \idwebr{0}{1.5}{1}
                \idwebr{0}{-1.5}{1}
                \idwebl{1}{-1.5}{1}
                \webarr{0}{3}
            \end{tikzpicture}
        }\xrightarrow{\tw}\lrb{
            \begin{tikzpicture}[anchorbase,scale=.3,rotate=90]
                \BSbl{0}{1.5}
                \idwebl{1}{0}{1}
                \idwebr{0}{0}{1}
                \idwebr{0}{-1.5}{1}
                \idwebl{1}{-1.5}{1}
                \webarr{0}{3}
            \end{tikzpicture}
        }\xrightarrow{\mathrm{RII}^{-1}}\lrb{
            \begin{tikzpicture}[anchorbase,scale=.3,rotate=90]
                \BSbl{0}{1.5}
                \webarr{0}{3}
            \end{tikzpicture}
        }.\]Then $\rho_0$ is a very strong deformation retract.
\end{prop}

\begin{proof}
    The following diagram depicts the $\pi=\rho_0$ and $\iota\coloneqq\tw^{-1}\circ\mathrm{RII}$ maps.\begin{equation}\label{eqn:vsdr 2 strands}
        \vsdr
    \end{equation}
    Moreover, we have $\iota\circ\pi-\idd=[d,H]$, where\[H\coloneqq\left(\vsdrhmtpy\right).\]It is routine to check that the side conditions hold for $H$.
\end{proof}

We are now ready to study the simplest nontrivial example of the flip symmetry: the flip map on one crossing, as shown in Figure~\ref{fig:flip2strands}.

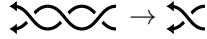
\begin{figure}[hbtp]
    \centering
    \[
        \begin{tikzpicture}[anchorbase,scale=.3,rotate=90]

            \idwebl{1}{1.5}{1}
            \idwebr{0}{1.5}{1}
            \idwebl{1}{0}{1}
            \idwebr{0}{0}{1}
            \idwebr{0}{-1.5}{1}
            \idwebl{1}{-1.5}{1}
            \webarr{0}{3}
        \end{tikzpicture}
        \to
        \begin{tikzpicture}[anchorbase,scale=.3,rotate=90]
            \idwebl{1}{1.5}{1}
            \idwebr{0}{1.5}{1}
            \webarr{0}{3}
        \end{tikzpicture}
    \]
    \caption{Conjugation by a positive twist gives the flip map on a single crossing.}
    \label{fig:flip2strands}
\end{figure}

From the perspective of the flip symmetry, there are distinguished crossings on both sides of Figure~\ref{fig:flip2strands}: the middle crossing in the source tangle and the unique crossing in the target tangle, respectively. Resolving these distinguished crossings yields $2$-step filtrations on link homologies. The subtle point here is that the most obvious map induced by the Reidemeister II move is not filtered. In some sense, this is because this map breaks the symmetry between the left and right crossings. To fix this, we first resolve the middle crossing in the source and run a homotopy perturbation argument similar to \cite[Lemma 4.5]{chen2025flip}.

\begin{prop}\label{prop:flip2strands}
    There exists a map \[\varphi\colon\lrb{
            \begin{tikzpicture}[anchorbase,scale=.3,rotate=90]

                \idwebl{1}{1.5}{1}
                \idwebr{0}{1.5}{1}
                \idwebl{1}{0}{1}
                \idwebr{0}{0}{1}
                \idwebr{0}{-1.5}{1}
                \idwebl{1}{-1.5}{1}
                \webarr{0}{3}
            \end{tikzpicture}
        }\to\lrb{
            \begin{tikzpicture}[anchorbase,scale=.3,rotate=90]
                \idwebl{1}{1.5}{1}
                \idwebr{0}{1.5}{1}
                \webarr{0}{3}
            \end{tikzpicture}
        }\]with the following properties:\begin{enumerate}
        \item $\varphi$ is chain homotopic to the map induced by the Reidemeister II move that eliminates two crossings.

        \item $\varphi$ is filtered with respect to the $2$-step filtrations.

        \item The filtered pieces of $\varphi$ (with respect to the aforementioned filtrations) are equal to the maps induced by the fork twists or Reidemeister II moves on corresponding subcomplexes.
    \end{enumerate}

    A similar statement holds for \[\varphi\colon\lrb{
            \begin{tikzpicture}[anchorbase,scale=.3,rotate=90]

                \idwebl{1}{1.5}{1}
                \idwebr{0}{1.5}{1}
                \idwebr{0}{0}{1}
                \idwebl{1}{0}{1}
                \idwebr{0}{-1.5}{1}
                \idwebl{1}{-1.5}{1}
                \webarr{0}{3}
            \end{tikzpicture}
        }\to\lrb{
            \begin{tikzpicture}[anchorbase,scale=.3,rotate=90]
                \idwebr{0}{1.5}{1}
                \idwebl{1}{1.5}{1}
                \webarr{0}{3}
            \end{tikzpicture}
        }\]with the difference that one of the filtered pieces is the negative of the fork twist.
\end{prop}

\begin{proof}

    We first prove the following diagram commutes up to homotopy.

    \begin{equation}\label{eqn:diagram_enddot}
        \twistenddot
    \end{equation}

    We do this by checking the composition $\mathrm{RII}\circ\begin{tikzpicture}[anchorbase,scale=.2]
            \draw[bl] (0,-0.01) to (0,1);
            \fill[bl] (0,1) circle (2.5mm);
        \end{tikzpicture}\circ\iota$ is equal (not just chain homotopic) to $\begin{tikzpicture}[anchorbase,scale=.2]
            \draw[bl] (0,-0.01) to (0,1);
            \fill[bl] (0,1) circle (2.5mm);
        \end{tikzpicture}$, where $\iota$ is the homotopy inverse of $\rho_0$ defined in (\ref{eqn:vsdr 2 strands}), as depicted in Figure~\ref{fig:commuteenddot}. The only nontrivial map in the composition is $\begin{tikzpicture}[anchorbase,scale=.2]
            \draw[bl] (0,-0.01) \pu (0,1) (0,1)\ur (1,1.8) (1,1.8)\rd (2,1) (2,1)\dl(1.5,0.5) (1.5,0.5)\lu(1,1);
            \fill[bl] (1,1) circle (2.5mm);
        \end{tikzpicture}=\begin{tikzpicture}[anchorbase,scale=.2]
            \draw[bl] (0,-0.01) to (0,1);
            \fill[bl] (0,1) circle (2.5mm);
        \end{tikzpicture}$. This is easier than checking that the original diagram (\ref{eqn:diagram_enddot}) commutes, because there is no room for a homotopy between two $1$-term complexes in the first row of (\ref{eqn:diagram_enddot}).

    \begin{figure}[hbtp]
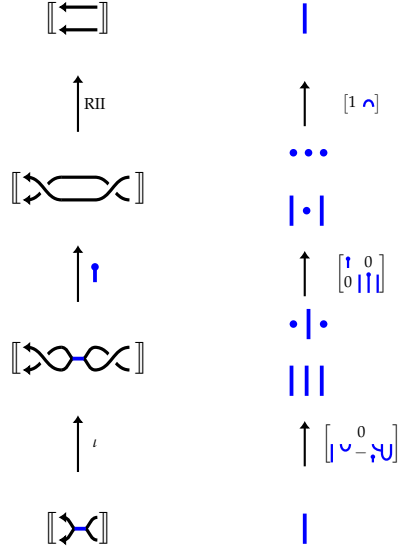

        \centering
        \[\commuteenddot\]
        \caption{Computing the composition of conjugating by the half twist for the end dot. Only resolutions in the relevant grading are depicted.}
        \label{fig:commuteenddot}
    \end{figure}

    By Proposition~\ref{prop:2strands vsdr} and \cite[Lemma 2.13]{willis2021khovanov}, the vertical maps in (\ref{eqn:diagram_enddot}) are very strong deformation retracts. So we can apply the homotopy perturbation lemma as in \cite[Lemma 4.5]{chen2025flip} to obtain a map $h$ such that the sum of the vertical maps in diagram (\ref{eqn:diagram_enddot_hmtpy}) is a chain map. Denote this map by $\varphi$. By definition, $\varphi$ satisfies properties (2) and (3).

    The tangles involved in the argument satisfy the assumption of Lemma~\ref{lem:simple tangle}, so $\varphi$ is chain homotopic to the original map induced by a Reidemeister II move up to sign. To see the sign is correct, we close the tangles up and look at the oriented resolutions, which support the homologies. It is then easy to see $\varphi$ is honestly chain homotopic to the map induced by a Reidemeister II move, so $\varphi$ also satisfies condition (1).

    \begin{equation}\label{eqn:diagram_enddot_hmtpy}
        \twistenddothmtpy
    \end{equation}

    The situation for negative crossings is similar. The only difference is that diagram (\ref{eqn:diagram_startdot}) is \textit{anti-commutative} instead of being commutative, so we have to replace $\rho_0$ by $-\rho_0$ to run the previous argument. \begin{equation}\label{eqn:diagram_startdot}
        \twiststartdot
    \end{equation}
\end{proof}

\section{The flip cobordism}\label{sec: flip n strand}

In this section, we turn to the global picture of the flip cobordism and use the maps studied in Section~\ref{sec:flip 2 strands} to construct a perturbed map $\rho'$ of $\rho$ that preserves the cubical grading.

Throughout this section, we assume that coherently oriented braids are presented horizontally in the plane, with all strands oriented from right to left. The (upper) closure of a coherently oriented braid is taken \textit{above} the braid and carries a natural orientation. One can also take the closure \textit{below} the braid, which we call the \textit{lower closure}. 

Let $D$ be a diagram of a link $K$ obtained as the closure of a coherently oriented braid $\beta$ with $n$ crossings and $k$ strands, and let $D^*$ be its flip diagram. Recall that there is a $\{0,1\}^n$-grading on both $\lrb{D}$ and $\lrb{D^*}$ given by remembering the $0$ or $1$ resolution at each crossing, which we call the \textit{cubical grading} hereafter. The main result of this section is the following.

\begin{thm}\label{thm: filtered map}
    The map \[\rho\colon\lrb{D}\to\lrb{D^*}\]induced by the flip cobordism is chain homotopic to a map $\rho'$ that preserves the cubical grading.
\end{thm}

To realize the flip cobordism as a sequence of Reidemeister moves, we introduce pairs of half twists and apply them to each crossing as in \cite{chen2025flip}. The construction below is a variant of \cite[Construction 4.3]{chen2025flip}. For us, a $(k,k)$ tangle is a $(2k)$-ended tangle with $k$ input boundary points and $k$ output boundary points, not necessarily oriented. We denote the Artin generators of the braid group on $k$ strands by $\sigma_1,\dots,\sigma_{k-1}$, so that the positive half twist on $k$ strands is presented as the braid word 
\[\Delta_{k}=(\sigma_{k-1}\sigma_{k-2}\cdots\sigma_1)(\sigma_{k-1}\cdots\sigma_{2})\cdots(\sigma_{k-1}\sigma_{k-2})\sigma_{k-1}.\]The following lemma is immediate.

\begin{lem}\label{lem flip 1 crossing tangles}
	Let $T$ be a $(k,k)$ tangle represented by a single braid word $\sigma_i$ (resp. $\sigma^{-1}_i$). Then $\Delta_{k}^{-1}\circ T\circ \Delta_{k}$ is isotopic rel boundary to $T^{*}= \sigma_{k-i}$ (resp. $\sigma_{k-i}^{-1}$) via Reidemeister II and III moves. Moreover, they are related by a cobordism between tangles that rotates the tangle about the $x$-axis by $180^\circ$. 
\end{lem}

Recall that given a $(1,1)$ tangle diagram $T$, presented horizontally, we have two ways to close $T$ up to obtain a link diagram: from the top or from the bottom. These two closures are related by the so-called \textit{half sweep-around move}, studied in \cite{morrison2022invariants}. It can be realized by two Reidemeister I moves and several Reidemeister II and III moves; see Figure~\ref{fig:half sweep around} for illustration.

\begin{figure}[htbp]
    \centering
    \includegraphics[width=0.8\linewidth]{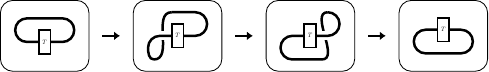}
    \caption{The half sweep-around move.}
    \label{fig:half sweep around}
\end{figure}

\begin{con}\label{con:flip cobordism}
    Assume that $D$ is the closure of a coherently oriented braid \[\beta=\sigma_{i_1}^{e_1}\circ\sigma_{i_2}^{e_2}\circ\cdots\circ\sigma_{i_n}^{e_n},\]where $i_1,\dots,i_n\in\{1,\dots,k-1\}$ and $e_1,\dots,e_n\in\{\pm 1\}$. Then the flip cobordism between $D$ and $D^*$ can be realized by the following sequence of moves; see Figure~\ref{fig:decompose flip cobordism} for illustration.
    
    \begin{enumerate}[i.]
    
    \item Insert $n$ canceling pairs of half twists $\Delta_{k} \circ \Delta_{k}^{-1}$ between $\sigma^{e_j}_{i_j}$ and $\sigma^{e_{j+1}}_{i_{j+1}}$ ($1\le j\le n$) via Reidemeister II moves. Here the subscripts are understood modulo $n$, so the last half twist is inserted between $\sigma^{e_n}_{i_n}$ and $\sigma^{e_1}_{i_1}$, which is valid after taking closure.
    
    \item Isotope $\Delta^{-1}_{k} \circ \sigma^{e_j}_{i_j} \circ \Delta_{k}$ to $\sigma^{e_j}_{k-i_j}$ for $1\le j\le n$ via Reidemeister moves described in Lemma~\ref{lem flip 1 crossing tangles}.
    
    \item From the outermost strand to the innermost strand, apply the half sweep-around move as in Figure~\ref{fig:half sweep around} (cf. \cite[Equation (3.1)]{morrison2022invariants}) $k$ times to move from the upper closure to the lower closure.
\end{enumerate}

\begin{figure}[hbtp]
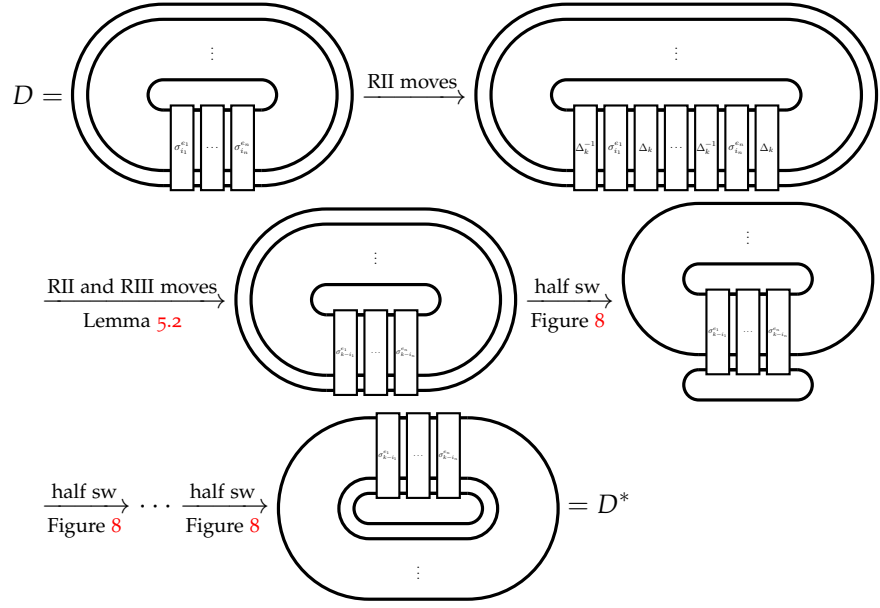

    
    \centering
    \begin{align*}
        D&=\flipcoba\xrightarrow{\text{RII moves}}\flipcobb\\&\xrightarrow[\text{Lemma~\ref{lem flip 1 crossing tangles}}]{\text{RII and RIII moves}}\flipcobc
        \xrightarrow[\text{Figure~\ref{fig:half sweep around}}]{\text{half sw}}\flipcobda\\&\xrightarrow[\text{Figure~\ref{fig:half sweep around}}]{\text{half sw}}\cdots\xrightarrow[\text{Figure~\ref{fig:half sweep around}}]{\text{half sw}}\flipcobdk=D^*
    \end{align*}
    \caption{Decompose the flip cobordism.}
    \label{fig:decompose flip cobordism}
\end{figure}

Schematically, the map can be depicted as follows: \[\lrb{D}\xrightarrow{\mathbb{1}\Rightarrow\Delta\Delta^{-1}}\lrb{\cdots\Delta^{-1}\sigma_{i_j}^{\pm}\Delta\cdots}\xrightarrow{\text{Lemma~\ref{lem flip 1 crossing tangles}}}\lrb{\sigma_{i_j}^{\pm *}}\xrightarrow{\text{half sw}}\lrb{D^*}.\]
\end{con}

The most interesting move in Construction~\ref{con:flip cobordism} is the second one. This move can be further localized near each crossing to a half twist conjugation. When $k=2$, this is exactly the scenario that we have studied in Section~\ref{sec:flip 2 strands}. The following proposition summarizes the necessary properties we need for general $k$ (cf. \cite[Lemma 4.5]{chen2025flip}).

\begin{prop}\label{prop:flip k strands}
    Let $T$ be a coherently oriented $k$-strand braid with a single crossing. Let \[\rho_T\colon\lrb{\Delta_k^{-1}\circ T\circ \Delta_k}\to\lrb{T^*}\] be a chain map induced by a sequence of Reidemeister moves\footnote{By Lemma~\ref{lem:simple tangle}, such a map is unique up to chain homotopy and sign.}. Then $\rho_T$ is chain homotopic to a filtered map $\rho_T'$ with respect to the $\{0,1\}$-grading coming from resolving the unique crossing in $T$.
\end{prop}

\begin{proof}
    Assume that the distinguished crossing in $T$ is between the $i$-th and $(i+1)$-th strands. The isotopy from $\Delta_k^{-1}\circ T\circ \Delta_k$ to $T^*$ can be performed in three steps; see Figure~\ref{fig:local flip twist} for an illustration when $k=4$ and $T=\sigma_1$. \begin{enumerate}[i.]
        \item Isotope the (three) crossings between the $i$-th and $(i+1)$-th strands to the left end of the braid via Reidemeister III moves. 
        \item Eliminate crossings outside the $i$-th and $(i+1)$-th strands via Reidemeister II moves.
        \item Eliminate two crossings between the $i$-th and $(i+1)$-th strands via a Reidemeister II move.
    \end{enumerate}

    \begin{figure}[hbtp]
        \centering
        \includegraphics{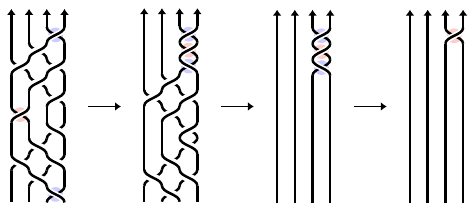}
        \caption{Decompose the local flip cobordism. The tangles are presented vertically instead of horizontally only for the sake of display. The distinguished crossing from $T$ is shaded by red, and two other crossings on the $i$-th and $(i+1)$-th strands are shaded by blue.}
        \label{fig:local flip twist}
    \end{figure}

There is a distinguished crossing in the diagram at each step, which induces a $\{0,1\}$-grading on the corresponding chain complex. As explained in \cite[Section 3.5]{morrison2022invariants} (see also \cite{elias2010rouquier}\cite[Section 3.2]{ren2025intrinsic}), the induced chain map in the first step can be chosen to be filtered with respect to the $\{0,1\}$-grading coming from the distinguished crossing. The second step does not involve the distinguished crossing, so the induced map is automatically filtered. The last step is exactly the situation studied in Section~\ref{sec:flip 2 strands}. By Proposition~\ref{prop:flip2strands}, we can replace the Reidemeister II map by a filtered map. Composing the three maps gives the desired filtered map $\rho_T'$.
\end{proof}

\begin{proof}[Proof of Theorem~\ref{thm: filtered map}]

We realize the map $\rho$ induced by the flip cobordism as the composition of the three moves in Construction~\ref{con:flip cobordism}. In Figure~\ref{fig:decompose flip cobordism}, each stage of the complex carries a $\{0,1\}^n$-grading inherited from the crossings originally in $D$, which we still call the cubical grading. It is clear that the first move in Construction~\ref{con:flip cobordism} does not involve the original crossings in $D$, so it preserves the cubical grading. The chain map induced by the third move in Construction~\ref{con:flip cobordism}, as explained in \cite[Section 3.5]{morrison2022invariants}, can also be chosen to be filtered with respect to the cubical grading. The second move in Construction~\ref{con:flip cobordism} is a composition of local moves as in Proposition~\ref{prop:flip k strands}, so the induced map can be replaced by a filtered map with respect to the cubical grading. Composing these three maps gives a filtered map $\rho'$ with respect to the cubical grading that is chain homotopic to $\rho$. Moreover, $\rho$ preserves the homological grading, so $\rho'$ also preserves the homological grading. Note that the homological grading is the sum of the entries of the cubical grading (up to an overall shift), so $\rho'$ must preserve the cubical grading. 
\end{proof}

\section{Planar webs}\label{sec:flip planar web}

In Section~\ref{sec: flip n strand}, we have reduced the computation of the flip map on a link diagram to a computation on planar $\gl_N$ webs. In this section, we study the flip map on planar webs and complete the proof of Theorem~\ref{thm:main}.

Following Section~\ref{sec: flip n strand}, let $D$ be a diagram of a link $K$ obtained as the closure of a coherently oriented braid with $n$ crossings, and let $D^*$ be its flip diagram. Let $v\in \{0,1\}^n$ be a vertex of the cube, let $\Gamma=D_v$ be the $\gl_N$ web arising at $v$ in the cube of resolutions of $D$, and let $\Gamma^*$ be the flip web of $\Gamma$ arising at the same vertex in the cube of resolutions of $D^*$. In Section~\ref{sec:bg flip}, we have defined an isomorphism \[\eta_\Gamma\colon\lrb{\Gamma}\to\lrb{\Gamma^*}\]by reflecting generating foams with signs corrected. There is another map \[\rho_\Gamma\colon\lrb{\Gamma}\to\lrb{\Gamma^*}\]arising as the restriction of the modified map $\rho'$ that we constructed in Section~\ref{sec: flip n strand} with a sign correction:\begin{equation}\label{eqn:def of rho_gamma}
\rho_\Gamma=(-1)^{n_-(v)}\rho'\big|_{\lrb{\Gamma}}.
\end{equation}Here $(-1)^{n_-(v)}$ arises from the sign appearing in Proposition~\ref{prop:flip2strands}. Our goal is to compare these two maps.

Tracing the construction of the modified map $\rho'$ in Theorem~\ref{thm: filtered map}, we see that the map $\rho_\Gamma$ is again a composition of three chain maps. The first one is the same as the first step in Construction~\ref{con:flip cobordism}, induced by Reidemeister II moves that introduce canceling pairs of half twists. The second one is a composition of the graded pieces of the map described in Proposition~\ref{prop:flip k strands}. The third one is essentially induced by the slide map studied in \cite{stroppel2024braiding}. In particular, the following definition records the atomic pieces of the second move in Construction~\ref{con:flip cobordism}, which is further a composition of slide maps in \cite{stroppel2024braiding} and a twist map in Definition~\ref{defn:fork twist}.

\begin{defn}\label{defn:n strand twist}
    Let $w$ be the web with one $2$-labeled edge that arises as a resolution of the $k$-strand braid $T$ in Proposition~\ref{prop:flip k strands}. The map \[\rho_{w}\colon\lrb{\Delta_k^{-1}\circ w\circ \Delta_k}\to\lrb{w^*}\]is defined as the filtered piece of the modified map $\rho'_T$ in Proposition~\ref{prop:flip k strands}, where $T$ has a single positive crossing.
\end{defn}

For a closed $\gl_N$ web $\Gamma$, the map $\rho_\Gamma$ can be obtained as a composition of several twist maps from Definition~\ref{defn:n strand twist} and slide maps from Proposition~\ref{prop:atomicslide}.

\begin{thm}\label{thm:flip planar web}
    Let $\Gamma$ be the $\gl_N$ web as above. Then $\rho_\Gamma=\eta_\Gamma$.
\end{thm}

\begin{rem}\label{rem:flip topological argument}
    Theorem~\ref{thm:flip planar web} is expected to be true (at least up to sign) by the following topological argument. Let $F\subset I\times \RR^3$ be the foam given by the trace of rotating $\Gamma$ around the $x$-axis. It can be thought of as a cobordism between $\Gamma$ and $\Gamma^*$, viewed as webs embedded in $\RR^3$. The map $\rho_\Gamma$ should be thought of as being induced by $F$. Now let $W$ be a foam in $\RR^3$ bounded by $\Gamma$; then $W\cup_\Gamma F$ is a foam bounded by $\Gamma^*$. Undoing the rotation yields an isotopy (as foams in $I\times\RR^3$) between $W\cup_\Gamma F$ and $W^*$, the foam in $\RR^3$ obtained by reflecting $W$ across the $xz$-plane. If we had a functorial $\gl_N$ theory for webs in $\RR^3$ and foams in $I\times \RR^3$, this argument would be enough to prove Theorem~\ref{thm:flip planar web}. In fact, this is how we compute the flip map for the crossingless unknot diagram in Example~\ref{ex:flip unknot}. Unfortunately, this level of functoriality seems to be currently unavailable beyond $\gl_2$ \cite{queffelec2022gl2}. Moreover, it seems that additional data such as cyclic orders and framings are required to fix the sign. We hope the present paper (in particular, the calculations in Section~\ref{sec:appendix}) can contribute to the development of such a functorial theory.
\end{rem}

Our strategy to prove Theorem~\ref{thm:flip planar web} is to establish several commutativity results between the half twist conjugation, the local model for the flip map, and local MOY foams. By Lemma~\ref{lem:MOY foam generates eco webs}, the state space of an elementary coherently oriented $\gl_N$ web is generated by MOY foams, so the commutativity results will allow us to reduce the computation to unlinks. The following proposition is the main technical result of this section, where we check the commutativity for each local piece of a MOY foam. 

\begin{prop}\label{prop:flip elementary gln foam}
    Let $\Gamma_0$ and $\Gamma_1$ be coherently oriented planar $\gl_N$ webs related by $W$, one of the local pieces of MOY foams introduced in Section~\ref{sec:bg MOY foams}. Let $\Delta$ be the positive half twist of the corresponding number of strands. Then the following diagram commutes up to homotopy and sign. 

    \[
        \tikzexternaldisable
        \begin{tikzcd}
            \lrb{\Gamma_0^*} \arrow[r, "W^*"]                & \lrb{\Gamma_1^*}\\
            \lrb{\Delta^{-1}\circ\Gamma_0\circ\Delta} \arrow[u, "\rho_{\Gamma_0}"'] \arrow[r, "W"] &  \lrb{\Delta^{-1}\circ\Gamma_1\circ\Delta}\arrow[u, "\rho_{\Gamma_1}"']
        \end{tikzcd}
        \tikzexternalenable
    \]

    The sign is $(-1)^{\delta(W)}$, where $\delta$ is defined in Section~\ref{sec:bg flip}.
\end{prop}

\begin{proof}
    There are several cases to check. The proof is a laborious but fairly straightforward computation, which we postpone to Section~\ref{sec:appendix}. Note that the signs are also checked there case by case.

    \begin{itemize}
        \item End dots. This is Lemma~\ref{lem:twist end}; in this case, $\delta(W)=0$.

        \item Start dots. This is Lemma~\ref{lem:twist start}; in this case, $\delta(W)=1$.

        \item Merge vertices. This is Lemma~\ref{lem:twist merge}; in this case, $\delta(W)=-1$.

        \item Split vertices. This is Lemma~\ref{lem:twist split}; in this case, $\delta(W)=0$.

        \item Six-valent vertices. These are Lemmas~\ref{lem:six valent} and~\ref{lem:six valenta}; in these cases, $\delta(W)=0$.

        \item Four-valent vertices. This is Lemma~\ref{lem:four valent}; in this case, $\delta(W)=0$.
    \end{itemize}
\end{proof}

\begin{proof}[Proof of Theorem~\ref{thm:flip planar web}]
    By Lemma~\ref{lem:MOY foam generates eco webs}, $\lrb{\Gamma}$ is generated by elements $[W\circ F_0]$, where $W$ is a MOY foam from an unlink $U_l$ to $\Gamma$ and $F_0\colon\emptyset\to U_l$ is a union of disks, possibly with dots. It suffices to show that they agree on such elements. For each MOY foam $W$ from an unlink $U_l$ to $\Gamma$, the following diagram commutes up to homotopy and sign by Proposition~\ref{prop:flip elementary gln foam} and \cite[Proposition 6.13]{stroppel2024braiding}: \[
        \tikzexternaldisable
        \begin{tikzcd}
            \lrb{U_l^*} \arrow[r, "W^*"]                & \lrb{\Gamma^*}\\
            \lrb{U_l} \arrow[u, "\rho_{U_l}"'] \arrow[r, "W"] &  \lrb{\Gamma}\arrow[u, "\rho_{\Gamma}"']
        \end{tikzcd}
        \tikzexternalenable
    \]
    All modules in the diagram above are concentrated in homological degree $0$, so the diagram commutes up to sign without homotopy. The sign is the product of the signs $(-1)^{\delta(W_k)}$ in Proposition~\ref{prop:flip elementary gln foam}, where $W_k$ ranges over the local pieces of $W$. By Lemma~\ref{lem:omega properties}(2), $\delta$ is additive under compositions and under planar compositions along $1$-labeled boundary points. Therefore, \[\rho_\Gamma\circ\lrb{W}=(-1)^{\delta(W)}\lrb{W^*}\circ\rho_{U_l}.\]By Example~\ref{ex:flip unknot}, the flip cobordism on unlinks sends $[F_0]$ to $[F_0^*]$. All labels in $F_0$ and $U_l$ are $1$, so $\omega(F_0)=\omega(U_l)=0$, and hence $\delta(W)=\omega(W)=\omega(W\circ F_0)$ by Lemma~\ref{lem:omega properties}(2). We conclude that \[\rho_\Gamma([W\circ F_0])=(-1)^{\omega(W\circ F_0)}[(W\circ F_0)^*]=\eta_\Gamma([W\circ F_0]).\]
\end{proof}

\begin{proof}[Proof of Theorem~\ref{thm:main}]
    When $D$ is obtained as the closure of a coherently oriented braid, Theorem~\ref{thm: filtered map} gives a representative $\rho'$ of $\rho$ that preserves the cubical grading, and $\rho'=\eta$ by Equation (\ref{eqn:def of rho_gamma}), the definition of $\eta$, and Theorem~\ref{thm:flip planar web}. Hence $\kappa=\eta^{-1}\circ\rho\simeq\idd$. In the general case, we choose a sequence of Reidemeister moves between $D$ and a coherently oriented braid closure $D_0$, which induces a chain homotopy equivalence $\varphi\colon\lrb{D}\to\lrb{D_0}$. Let $\kappa_D$ and $\kappa_{D_0}$ be the flip maps of $D$ and $D_0$. By Proposition~\ref{prop:flip invariance} and the braid closure case, we have $\varphi\circ\kappa_D\simeq\kappa_{D_0}\circ\varphi\simeq\varphi$. Since $\varphi$ is a chain homotopy equivalence, $\kappa_D\simeq\idd$.
\end{proof}

\begin{rem}\label{rem: half twists on Soergel mod}
The naturality result above can also be interpreted in the language of Soergel bimodules, which may be of independent interest. Let $\mathrm{SBim}_n$ be the category of Soergel bimodules of type $A_{n-1}$, whose bounded homotopy category $K^b(\mathrm{SBim}_n)$ carries a braid group action by Rouquier complexes. The positive half twist $\Delta_n$ defines an invertible object of $K^b(\mathrm{SBim}_n)$, and conjugation by it defines a monoidal autoequivalence \[X\mapsto \lrb{\Delta_n^{-1}}\otimes X\otimes\lrb{\Delta_n}.\]On the other hand, the nontrivial automorphism $i\mapsto n-i$ of the Dynkin diagram of type $A_{n-1}$ induces another monoidal autoequivalence, sending the Bott--Samelson generator $B_i$ to $B_{n-i}$. These two autoequivalences already agree on objects: the half twist $\Delta_n$ lifts the longest element $w_0\in S_n$, whose conjugation action $s_i\mapsto s_{n-i}$ is precisely the diagram automorphism. Our computations (Proposition~\ref{prop:flip elementary gln foam}) upgrade this to a natural isomorphism between the two autoequivalences, compatible with the generating morphisms of Soergel calculus, up to signs that we determine explicitly. In this sense, the flip symmetry is essentially a consequence of the Dynkin diagram involution, cf. \cite[Remark 3.18]{queffelec2016sln}.
\end{rem}

\section{Strongly invertible knots}\label{sec:strongly invertible}

In this section, we define the maps $\tau_t$ and $\tau_i$ in Corollary~\ref{cor:strongly invertible} and prove the corollary. The definition involves the reordering isomorphisms from Section~\ref{sec:bg gln homology} and the reversal of orientations, which only matter for the signs in the definitions.

We first briefly discuss the behavior of link homology under orientation reversal. For a web or foam $\Sigma$, let $\Sigma^{\mathrm{op}}$ be $\Sigma$ with the orientations of all edges, facets and seams reversed, keeping the labels and decorations. For a link diagram $D$, let $-D$ be $D$ with all orientations reversed. The crossings of $-D$ have the same signs as those of $D$, and $(-D)_v=(D_v)^{\mathrm{op}}$ for all $v$.

\begin{lem}\label{lem:op evaluation}
    Let $F$ be a closed $\gl_N$ foam. Then $\langle F^{\mathrm{op}}\rangle=(-1)^{\omega(F)}\langle F\rangle$.
\end{lem}

\begin{proof}
    The colorings of $F^{\mathrm{op}}$ are canonically identified with those of $F$. For each coloring, the polynomials $P$ in the Robert--Wagner evaluation formula agree, and the monochrome and bichrome surfaces agree as subsets, so $Q$ and the first term of $s$ agree. The cyclic order of the facets around a seam is determined by the orientation of the seam via the left-hand rule \cite{robert2020closed}, so reversing the orientations of all seams reverses all cyclic orders. Therefore, positive curves of $F^{\mathrm{op}}$ are negative curves of $F$, and the rest of the proof of Lemma~\ref{lem:reflecting closed foams} applies verbatim.
\end{proof}

Similar to Proposition~\ref{prop:signed reflection}, the assignments $\Gamma\mapsto\Gamma^{\mathrm{op}}$ and $H\mapsto H^\ddagger\coloneqq(-1)^{\delta(H)}H^{\mathrm{op}}$ define a functor from $\FoamN$ to itself with $(H^\ddagger)^\ddagger=H$, which is compatible with planar compositions along $1$-labeled boundary points. For a link diagram $D$, define \[\eta^{\mathrm{op}}_D\colon\lrb{D}\to\lrb{-D},\qquad[F]\mapsto(-1)^{\omega(F)+n_-(v)}[F^{\mathrm{op}}]\quad\text{for }[F]\in\lrb{D_v}.\]As for $\eta$, the reordering isomorphisms commute with $\eta^{\mathrm{op}}_D$.

\begin{prop}\label{prop:orientation reversal}
    The map $\eta^{\mathrm{op}}_D$ is a chain isomorphism with inverse $\eta^{\mathrm{op}}_{-D}$. Moreover, it is compatible with Reidemeister moves. More precisely, let $D_1$ and $D_2$ be link diagrams related by a Reidemeister move, and let $f\colon\lrb{D_1}\to\lrb{D_2}$ be the induced chain map. Let $f^-\colon\lrb{-D_1}\to\lrb{-D_2}$ be the chain map induced by the same move with all orientations reversed. Then the following diagram commutes up to homotopy:
    \[
        \tikzexternaldisable
        \begin{tikzcd}
            \lrb{D_1} \arrow[r, "\eta^{\mathrm{op}}_{D_1}"] \arrow[d, "f"] & \lrb{-D_1} \arrow[d, "f^-"] \\
                \lrb{D_2} \arrow[r, "\eta^{\mathrm{op}}_{D_2}"]                & \lrb{-D_2}
            \end{tikzcd}.
        \tikzexternalenable
    \]
\end{prop}

\begin{proof}
    The zip and unzip foams of $-D$ are the orientation reversals of those of $D$, and $\delta(H^{\mathrm{op}})=\delta(H)$. Hence the proof of Lemma~\ref{lem:eta chain map} shows that $\eta^{\mathrm{op}}_D$ is a chain map, and $\eta^{\mathrm{op}}_{-D}\circ\eta^{\mathrm{op}}_D=\idd$ since $(F^{\mathrm{op}})^{\mathrm{op}}=F$. For the second statement, the proof of Proposition~\ref{prop:eta reidemeister} applies with $H\mapsto H^\dagger$ and $\eta_D$ replaced by $H\mapsto H^\ddagger$ and $\eta^{\mathrm{op}}_D$. 
\end{proof}

Now let $(K,\tau)$ be a strongly invertible knot with axis $\ell$. Let $D_t$ be a transvergent diagram of $(K,\tau)$, and assume further that $\ell$ is the $x$-axis of the projection plane. The flip diagram $D_t^*$, with the ordering of crossings inherited from $D_t$, agrees with $-D_t$ up to the ordering of crossings, so there is a unique permutation $\sigma_t$ with \[-D_t=(D_t^*)^{\sigma_t}.\]Similarly, let $D_i$ be an intravergent diagram of $(K,\tau)$, and let $\gamma$ be the rotation of the projection plane induced by the rotation by $\pi$ around $\ell$. Then there is a unique permutation $\sigma_i$ with \[-D_i=\gamma(D_i)^{\sigma_i}.\]The planar isotopy $\gamma$ induces an isomorphism \[\gamma_*\colon\lrb{D_i}\to\lrb{\gamma(D_i)},\qquad[F]\mapsto[\gamma(F)],\]where $\gamma$ acts on foams by rotating the plane and fixing the vertical coordinate. Since $\tau$ is an involution, $\sigma_t^2=\sigma_i^2=\idd$, and both permutations preserve the signs of crossings.

\begin{defn}\label{defn:strong involutions}
    The chain maps $\tau_t\colon\lrb{D_t}\to\lrb{D_t}$ and $\tau_i\colon\lrb{D_i}\to\lrb{D_i}$ are defined as compositions \[\tau_t\coloneqq\eta^{\mathrm{op}}_{-D_t}\circ P_{\sigma_t}\circ\eta_{D_t},\ \tau_i\coloneqq\eta^{\mathrm{op}}_{-D_i}\circ P_{\sigma_i}\circ\gamma_*.\]More explicitly, for $[F]\in\lrb{(D_t)_v}$ and $[G]\in\lrb{(D_i)_v}$, we have \[\tau_t([F])=(-1)^{\operatorname{inv}_{\sigma_t}(v)+\operatorname{inv}_{\sigma_t}(o)}[(F^*)^{\mathrm{op}}],\ \tau_i([G])=(-1)^{\operatorname{inv}_{\sigma_i}(v)+\operatorname{inv}_{\sigma_i}(o)+\omega(G)+n_-(v)}[\gamma(G)^{\mathrm{op}}],\]which lie in the summands of $\lrb{D_t}$ and $\lrb{D_i}$ at the vertices $\sigma_tv$ and $\sigma_iv$, respectively. 
\end{defn}

Here the signs in $\tau_t$ from $\eta_{D_t}$ and $\eta^{\mathrm{op}}_{-D_t}$ cancel since $\omega(F^*)=\omega(F)$ and $n_-(\sigma_tv)=n_-(v)$. It is clear that $\tau_t$ and $\tau_i$ are involutions. Over $\FF_2$, they are just obtained by applying the symmetry on each resolution and reversing orientations.

We regard link diagrams as links in $\RR^3$ near the projection plane, and fix an orientation of $\ell$. Let $S_t$ be the trace of the rotation of $D_t$ by $\pi$ around the $x$-axis, and let $S_i$ be the trace of the rotation of $D_i$ by $\pi$ around the axis perpendicular to the projection plane through the center of $\gamma$, both in the positive direction with respect to the orientation of $\ell$. These are cobordisms from $D_t$ to $-D_t$ and from $D_i$ to $-D_i$, respectively. For the trace $W$ of an isotopy from $D_t$ to $D_i$, let $-W$ be the trace of the same isotopy from $-D_t$ to $-D_i$. We say that $W$ is \textit{equivariant up to isotopy} if the following diagram of cobordisms commutes up to isotopy relative to the boundary:
    \[
        \tikzexternaldisable
            \begin{tikzcd}
                D_t \arrow[r, "S_t"] \arrow[d, "W"] & -D_t \arrow[d, "-W"] \\
                D_i \arrow[r, "S_i"]                & -D_i
            \end{tikzcd}.
        \tikzexternalenable
    \]
In this case, let $\varphi\colon\lrb{D_t}\to\lrb{D_i}$ be the induced chain homotopy equivalence. Such an isotopy always exists. For instance, we can take the trace of a rigid motion of $\RR^3$ carrying the axis of $D_t$ to that of $D_i$, followed by an isotopy commuting with the rotation around the latter.

\begin{proof}[Proof of Corollary~\ref{cor:strongly invertible}]

    Let $W$ be a trace of an isotopy from $D_t$ to $D_i$, equivariant up to isotopy, so $\varphi=\lrb{W}$. The idea is to express $\tau_t$ and $\tau_i$ as maps induced by cobordisms and use functoriality. By definition, $S_t$ is the flip cobordism of $D_t$. By Theorem~\ref{thm:main}, we have \[\lrb{S_t}\simeq P_{\sigma_t}\circ\rho\simeq P_{\sigma_t}\circ\eta_{D_t}\]since $-D_t$ carries the ordering of crossings inherited from $D_t$. Similarly, $S_i$ is the trace of the planar rotation $\gamma$, so \[\lrb{S_i}\simeq P_{\sigma_i}\circ\gamma_*\]since $-D_i$ carries the ordering of crossings inherited from $D_i$. By Proposition~\ref{prop:orientation reversal}, we have \[\eta^{\mathrm{op}}_{-D_i}\circ\lrb{-W}\simeq\lrb{W}\circ\eta^{\mathrm{op}}_{-D_t}.\]The compositions of the rows in the following diagram are chain homotopic to $\tau_t$ and $\tau_i$, respectively:
    \[
        \tikzexternaldisable
            {\begin{tikzcd}
                \lrb{D_t} \arrow[r, "\lrb{S_t}"] \arrow[d, "\varphi"] & \lrb{-D_t} \arrow[r, "\eta^{\mathrm{op}}_{-D_t}"] \arrow[d, "\lrb{-W}"] & \lrb{D_t} \arrow[d, "\varphi"] \\
                \lrb{D_i} \arrow[r, "\lrb{S_i}"]                      & \lrb{-D_i} \arrow[r, "\eta^{\mathrm{op}}_{-D_i}"]                       & \lrb{D_i}
            \end{tikzcd}}.
        \tikzexternalenable
    \]
    The left square commutes up to homotopy since $W$ is equivariant up to isotopy. The movie of $-W$ from $-D_t$ to $-D_i$ becomes the movie of $W$ after reversing orientations, so the right square commutes up to homotopy by Proposition~\ref{prop:orientation reversal}. Therefore, we have \[\tau_i\circ\varphi\simeq\varphi\circ\tau_t.\]
\end{proof}

\section{Naturality of the half twist slide}\label{sec:appendix}

In this section, we present the diagrammatic calculations that constitute the heart of the proof of Proposition~\ref{prop:flip elementary gln foam}. These calculations also establish the natural isomorphism described in Remark~\ref{rem: half twists on Soergel mod}.

\subsection{End and start dots}

\begin{lem}\label{lem:twist end}
    The following diagram commutes up to homotopy.
    \[\halftwistenddot\]
\end{lem}

\begin{proof}
    This is the diagram (\ref{eqn:diagram_enddot}), which we proved to be commutative in the proof of Proposition~\ref{prop:flip2strands}.
\end{proof}

\begin{lem}\label{lem:twist start}
    The following diagram commutes up to homotopy.
    \begin{equation*}
        \halftwiststartdot
    \end{equation*}
\end{lem}

\begin{proof}
    This is essentially the same diagram as (\ref{eqn:diagram_startdot}), which we have proved to be anti-commutative in the proof of Proposition~\ref{prop:flip2strands}. The sign there is absorbed here.
\end{proof}

\subsection{Merge and split vertices}

\begin{lem}\label{lem:twist merge}
    The following diagram commutes up to homotopy.
    \begin{equation*}
        \halftwistmerge
    \end{equation*}
\end{lem}

\begin{proof}
    We prove that the following diagram commutes (with trivial homotopy).

    \begin{equation*}
        \halftwistmergeproof
    \end{equation*}

    In fact, we have
    \begin{gather*}
        \left(\halftwistmergea\right)-\left(\halftwistmergeb\right)\\=\left(\halftwistmergec\right)=0.
    \end{gather*}
\end{proof}

\begin{lem}\label{lem:twist split}
    The following diagram commutes up to homotopy.
    \begin{equation*}
        \halftwistsplit
    \end{equation*}
\end{lem}

\begin{proof}
    This is similar to the proof of Lemma~\ref{lem:twist merge}.
\end{proof}

\subsection{Six-valent vertices}

To compute the half twist slide through a six-valent vertex, we first need to write down the map induced by sliding a half twist on $3$ strands through a thick edge. Recall that our notation for half twists on three strands is\[\Delta_3=\begin{tikzpicture}[anchorbase,scale=.25,rotate=90]
        \deltathree{0}{0}
        \webarrr{0}{3}
    \end{tikzpicture},\,\Delta_3^{-1}=\begin{tikzpicture}[anchorbase,scale=.25,rotate=90]
        \deltathreei{0}{0}
        \webarrr{0}{3}
    \end{tikzpicture},\,\Delta_3'=\begin{tikzpicture}[anchorbase,scale=.25,rotate=90]
        \deltathreea{0}{0}
        \webarrr{0}{3}
    \end{tikzpicture},\,\Delta_3'^{-1}=\begin{tikzpicture}[anchorbase,scale=.25,rotate=90]
        \deltathreeai{0}{0}
        \webarrr{0}{3}
    \end{tikzpicture}.\]

\begin{prop}\label{prop:slide3strandtwist}
    The map induced by sliding the $3$-strand half twist $\Delta_3$ through a red thick edge, given by a composition of the slide map and the twist map described in Proposition~\ref{prop:atomicslide} and Definition~\ref{defn:fork twist} respectively, is depicted in Figure~\ref{fig:slide_ht3r_ltr}. The inverse of this map is depicted in Figure~\ref{fig:slide_ht3r_rtl}. The map induced by sliding the $3$-strand half twist $\Delta_3'$ through a blue thick edge can be obtained by switching the red color and the blue color in Figure~\ref{fig:slide_ht3r_ltr}; its inverse is given by switching colors in Figure~\ref{fig:slide_ht3r_rtl}.
\end{prop}

\begin{figure}[hbtp]
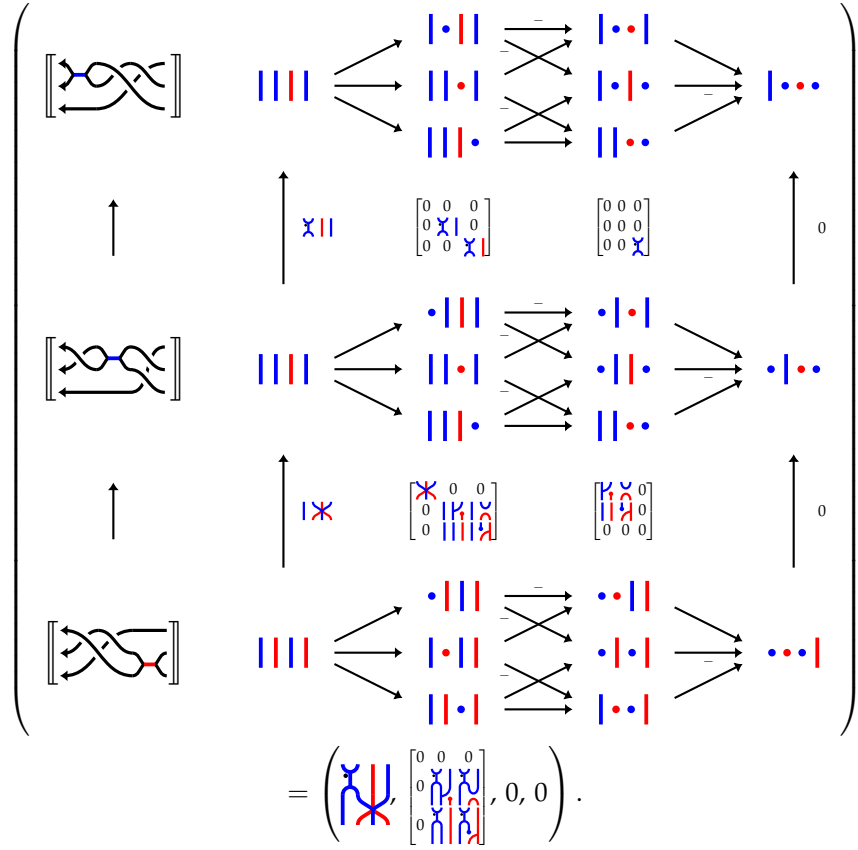

    \centering
    \begin{gather*}
        \left(
        \slidehtrltr\right)\\
        =\left(\slidehttltra,\,\slidehttltrb,\,0,\,0\right).\end{gather*}
    \caption{Slide the half twist through a red edge, left to right.}
    \label{fig:slide_ht3r_ltr}
\end{figure}

\begin{figure}[hbtp]
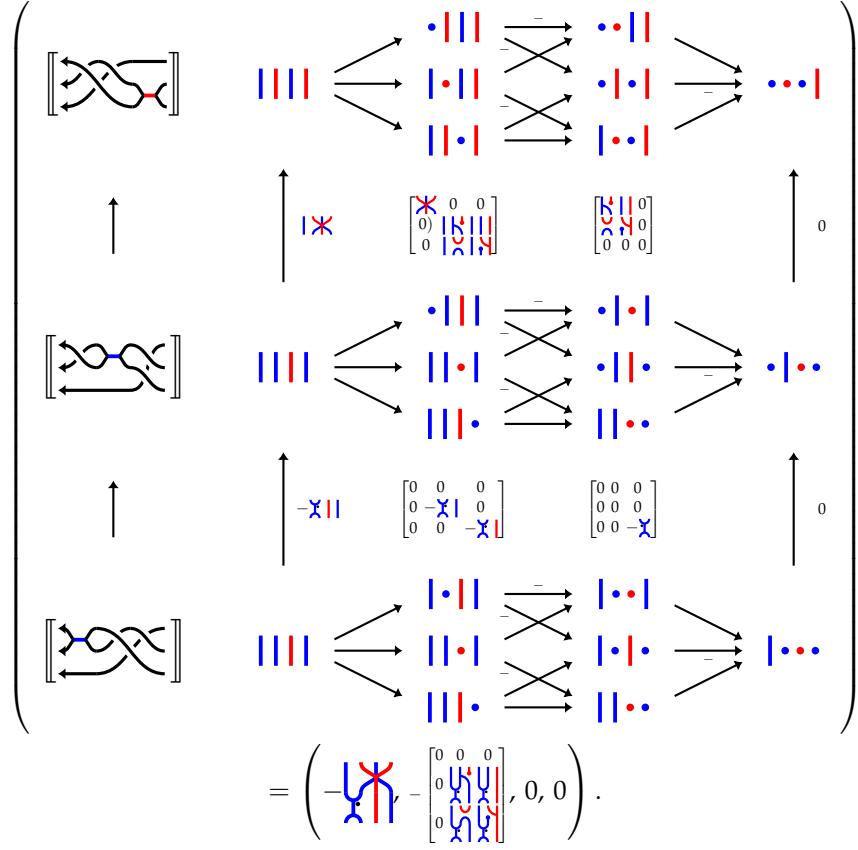

    \centering
    \begin{gather*}
        \left(
        \slidehtrrtl\right)\\
        =\left(\slidehttltrai,\,\slidehttltrbi,\,0,\,0\right).\end{gather*}
    \caption{The inverse of the map in Figure~\ref{fig:slide_ht3r_ltr}.}
    \label{fig:slide_ht3r_rtl}
\end{figure}

\begin{lem}\label{lem:six valent}
    The following diagram commutes up to homotopy.
    \[
        \halftwiststar
    \]
\end{lem}

\begin{proof}
    This is the most involved computation in this paper. The plan is again to check that the composition map from the top left term to the top right term is the same as the map $\ssixv$, without worrying about the possible homotopy. The movie realizing this composition is depicted in Figure~\ref{fig:movieconjstar}.

    \begin{figure}[hbtp]
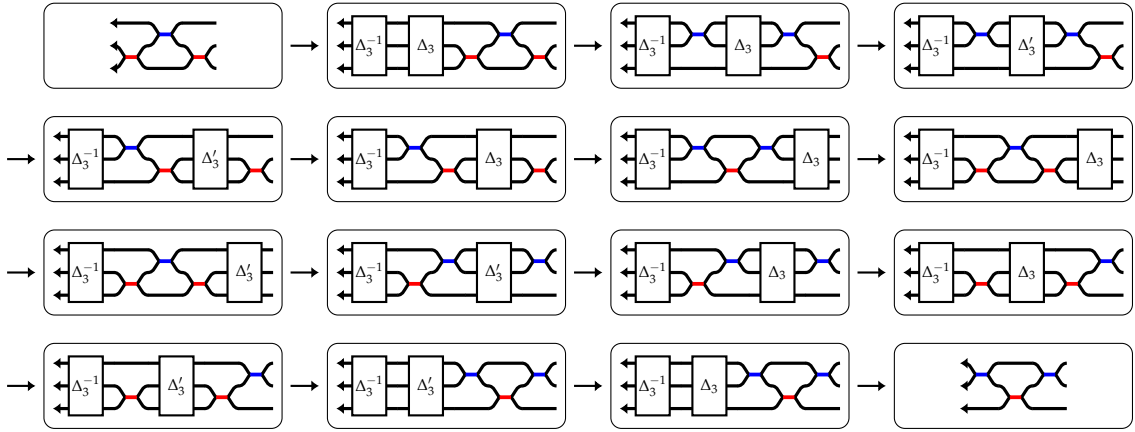

        \centering
        \[\movestar\]
        \caption{Conjugating a six-valent vertex by the half twist.}
        \label{fig:movieconjstar}
    \end{figure}

    There are four types of moves in Figure~\ref{fig:movieconjstar}:\begin{itemize}
        \item Reidemeister II moves (Frame 1 to 2, and 15 to 16).

        \item Reidemeister III moves between $\Delta_3$ and $\Delta_3'$ (Frame 3 to 4, 5 to 6, 8 to 9, 10 to 11, 12 to 13, and 14 to 15).

        \item The six-valent vertex (Frame 7 to 8).

        \item The moves that slide the half twists through thick edges (Frame 2 to 3, 4 to 5, 6 to 7, 9 to 10, 11 to 12, 13 to 14).
    \end{itemize}

    The induced maps of Reidemeister II and III moves have been described in Propositions~\ref{prop:RIImove} and~\ref{prop:RIIImove}. Moves of the last type, sliding half twists, are described in Proposition~\ref{prop:slide3strandtwist}.

    From Frame 2 to Frame 15, three crossings in $\Delta_3^{-1}$ do not participate in the movie. By Proposition~\ref{prop:RIImove}, a resolution of six crossings in the second frame must be \textit{palindromic} to support a nonzero image from Frame 1. Moreover, the effect of the movie from the first frame to the second and from the second-to-last frame to the last one is to take a ``partial trace'' on all but the two rightmost strands in Soergel calculus diagrams, up to sign.

    For the chain complexes associated to Frame 2 to Frame 15, we define an auxiliary grading $h_3\in\{0,1,2,3\}$ as the number of $1$-resolutions when resolving three crossings in $\Delta_3$ or $\Delta_3'$. It is clear from Propositions~\ref{prop:RIIImove} and~\ref{prop:slide3strandtwist} that the map induced by the movie from Frame 2 to Frame 15 preserves the $h_3$ grading. Denote the map induced by the movie in Figure~\ref{fig:movieconjstar} in $h_3$ grading $i$ by $f_i$ ($i=0,\,1,\,2,\,3$). It is easy to see from Proposition~\ref{prop:slide3strandtwist} that $f_2$ and $f_3$ vanish. It remains to compute $f_0$ and $f_1$; this is done in Lemma~\ref{lem:h3=0} for $f_0$ and Lemma~\ref{lem:h3=1} for $f_1$. In conclusion, the map induced by the movie in Figure~\ref{fig:movieconjstar} is given by \[f_0+f_1=2\ssixv-\ssixv=\ssixv.\]\end{proof}

\begin{lem}\label{lem:h3=0}
    The $h_3=0$ piece of the map induced by the movie in Figure~\ref{fig:movieconjstar} is given by $f_0=2\ssixv$.
\end{lem}

\begin{proof}

    The map $f_0$ is a composition of $15$ maps. From Frame 2 to Frame 15, each map between two frames is a single-term map (before expanding the black dot notation). Direct concatenation gives\[f_0=-\hugestarh,\]where the sign comes from taking the Reidemeister II move (that eliminates crossings) three times, and the diagram should be understood as being rotated by $90^\circ$ clockwise.

    Before starting to simplify this huge diagram, we first note that \begin{equation}\label{eqn:blackdot_rotate}
        \begin{tikzpicture}[anchorbase,scale=.2]
            \gonglb{0}{0}
        \end{tikzpicture}=-\begin{tikzpicture}[anchorbase,scale=.2,rotate=90]
            \gongrb{0}{0}
        \end{tikzpicture}.
    \end{equation}
    Using this, we can turn this diagram into a symmetric form, as depicted in Figure~\ref{fig:hugestarsymmetric}. The sign therein comes from applying the relation (\ref{eqn:blackdot_rotate}) three times (to the top three black dots).
    \begin{figure}[hbtp]
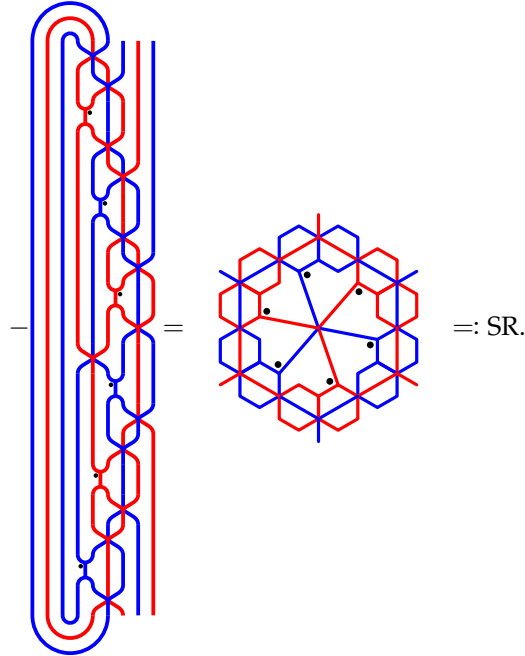

        \centering
        \begin{gather*}
            -\hugestarv=\hugestarsymmetric\eqqcolon \mathrm{SR}.
        \end{gather*}
        \caption{Isotope the $h_3=0$ piece of the map to a symmetric position.}
        \label{fig:hugestarsymmetric}
    \end{figure}

    It remains to check $\mathrm{SR}=2\ssixv$. Recall from (\ref{eqn:defn_black_dot}) that a black dot is a shorthand notation for two diagrams. So the total expansion would contain $2^6=64$ terms. Fortunately, the diagram has a $C_6$ symmetry given by the $C_3$ symmetry from rotating $\mathrm{SR}$ by $120^\circ$ and the $C_2$ symmetry from swapping colors. Modulo this symmetry, there are $14$ cases to check.

    We do not present the details of this case-by-case computation, as it is routine and unenlightening. We only record the result of this computation. We call the negative summand in (\ref{eqn:defn_black_dot}) the \textit{a-resolution}, and the positive summand in (\ref{eqn:defn_black_dot}) the \textit{b-resolution}. It turns out that each resolution is equal to a number of copies of the six-valent vertex $\ssixv$. Table~\ref{tab:computing huge star} provides the results: the total number of $\ssixv$ is $2$, as desired.

    \begin{table}[htbp]
        \centering

        \begin{tabular}{cccc}
            Resolution               & number of $\ssixv$ & number of resolutions & sign \\ \hline
            aaaaaa                   & $6$                & $1$                   & $+$    \\
            aaaaab                   & $3$                & $6$                   & $-$    \\
            aaaabb                   & $2$                & $6$                   & $+$    \\
            aaabab                   & $3$                & $6$                   & $+$    \\
            aabaab                   & $2$                & $3$                   & $+$    \\
            aaabbb                   & $2$                & $6$                   & $-$    \\
            aababb                   & $2$                & $6$                   & $-$    \\
            abaabb                   & $2$                & $6$                   & $-$    \\
            ababab                   & $3$                & $2$                   & $-$    \\
            aabbbb                   & $2$                & $6$                   & $+$    \\
            ababbb                   & $2$                & $6$                   & $+$    \\
            abbabb                   & $2$                & $3$                   & $+$    \\
            abbbbb                   & $2$                & $6$                   & $-$    \\
            bbbbbb                   & $2$                & $1$                   & $+$    \\\hline
            Total number of $\ssixv$ &                    &                       & $2$
            \\\hline
        \end{tabular}

        \caption{Computation result for summands of the diagram $\mathrm{SR}$.}
        \label{tab:computing huge star}
    \end{table}
\end{proof}

\begin{lem}\label{lem:h3=1}
    The $h_3=1$ piece of the map induced by the movie in Figure~\ref{fig:movieconjstar} is given by $f_1=-\ssixv$.
\end{lem}

\begin{proof}
    As in the previous lemma, the map $f_1$ is a composition of $15$ maps, but this time each map is a $3\times 3$ matrix from Frame 2 to Frame 15 (except from Frame 7 to Frame 8, which is a single six-valent vertex). Composing them still gives a $3\times 3$ matrix, and the effect of pre- and post-composing the Reidemeister II moves is to take the trace of this matrix; there is no sign change, since two caps occur this time.

    Our strategy for this computation is to first compute the composition of a Reidemeister III move and a half twist slide move; the total map from Frame 2 to Frame 15 is a composition of $3$ compositions of this type, a six-valent vertex, and another $3$ compositions. Computation gives \[f_1=\operatorname{tr}\left(\foneleft\circ\begin{tikzpicture}[anchorbase,scale=.2]
                \sixvbrb{0}{-1}
                \Rp{3}{0}
                \Rp{4}{0}
            \end{tikzpicture}\circ\foneright\right).\]Here $\operatorname{tr}$ refers to taking the trace of this matrix and closing the two leftmost strands up, and $\bullet$ indicates terms that are irrelevant to the trace.

    The trace is the sum of three terms on the diagonal: we have \[f_1=0+\left(-\fonetraa+\fonetrab\right)+\left(\fonetrba-\fonetrbb\right)=-\ssixv\ .\]Here all terms but the last one vanish because each of them contains a closed loop bounding a disk, which therefore evaluates to zero, and the last term is equal to $\ssixv$. This concludes the proof.
\end{proof}

\begin{lem}[The second six-valent vertices]\label{lem:six valenta}
    The following diagram commutes up to homotopy.
    \begin{equation*}
        \halftwiststara
    \end{equation*}
\end{lem}

\begin{proof}
    The diagram is a composition of two diagrams. \begin{equation*}
        \halftwiststaraproof
    \end{equation*}The bottom half one clearly commutes up to homotopy, and the top half one commutes up to homotopy by Lemma~\ref{lem:six valent} after swapping colors.
\end{proof}

\subsection{Four-valent vertices}

\begin{lem}\label{lem:four valent}
    The following diagram commutes up to homotopy.
    \begin{equation*}
        \halftwistcross
    \end{equation*}
\end{lem}

\begin{proof}
    The half twist conjugation on four strands is a composition of two operations: the conjugation action of the $(2,2)$ shuffle braid and the conjugation action of two separate half twists on the top two and bottom two strands, as depicted in Figure~\ref{fig:decompose 4 strand half twist}. \begin{figure}[hbtp]
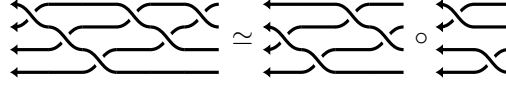

        \centering
        \[\halftwistfourstrands\simeq\deltatwotwo\circ\doublefork\]
        \caption{Decomposing the half twist conjugation on four strands into two operations.}
        \label{fig:decompose 4 strand half twist}
    \end{figure}It then suffices to prove that these two operations commute with the four-valent vertices up to homotopy, which is done in Lemmas~\ref{lem:shuffle_vertex} and~\ref{lem:double_fork_twist}, respectively.
\end{proof}

To establish the commutativity of the four-valent vertex and the $(2,2)$ shuffle braid, we first need to write down the map induced by sliding a $(2,2)$ shuffle braid through a thick edge.

\begin{prop}\label{prop:slide22shuffle}
    The map induced by sliding the $(2,2)$ shuffle braid $\Delta_{2,2}$ through a green thick edge, given by a composition of two slide maps described in Proposition~\ref{prop:atomicslide}, is depicted in Figure~\ref{fig:slide_sf22_ltr}. The inverse of this map is depicted in Figure~\ref{fig:slide_sf22_rtl}. The map induced by sliding another $(2,2)$ shuffle braid $\Delta_{2,2}'$ through a red thick edge can be obtained by switching the red color and the green color in Figure~\ref{fig:slide_sf22_ltr}; its inverse is given by switching colors in Figure~\ref{fig:slide_sf22_rtl}.
\end{prop}

\begin{proof}
    This is a direct computation using Proposition~\ref{prop:atomicslide}. The last two terms are necessarily zero because the map in the highest homological degree in Proposition~\ref{prop:atomicslide} vanishes, so we do not spell out the intermediate maps in the figures.
\end{proof}

\begin{figure}[hbtp]
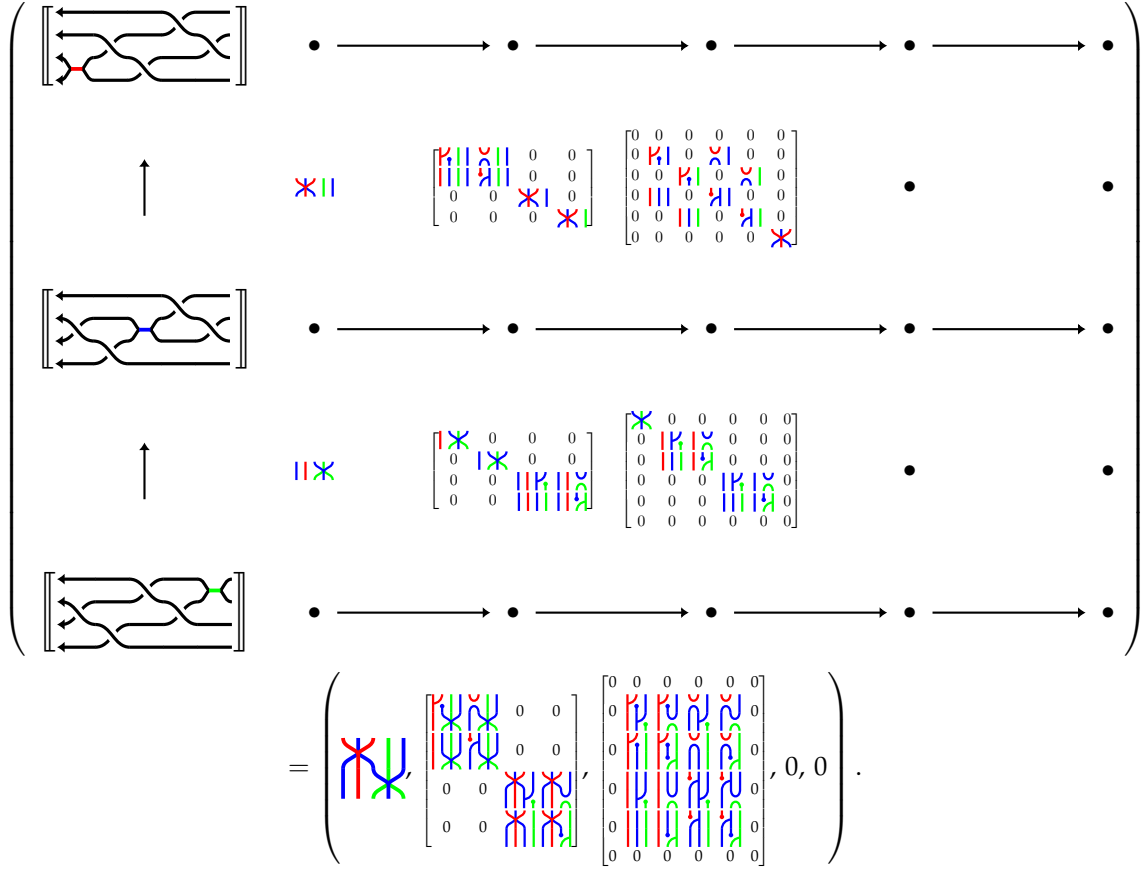

    \centering
    \tikzexternaldisable
    \begin{gather*}
        \left(
        \slidesfttltr\right)\\
        =\left(\slidesfttltra,\,\slidesfttltrb,\,\slidesfttltrc,\,0,\,0\right).
        \end{gather*}
    \tikzexternalenable
    \caption{Slide the $(2,2)$ shuffle braid $\Delta_{2,2}$ through a green edge, left to right. In the figure, bullets in the lines indicate the chain groups that can be read off from the context; bullets between the lines indicate the chain maps that are irrelevant to the final computation.}
    \label{fig:slide_sf22_ltr}
\end{figure}

\begin{figure}[hbtp]
    \centering
    \tikzexternaldisable
    \begin{gather*}
        \left(
        \slidesfttrtl\right)\\
        =\left(\slidesfttrtla,\,\slidesfttrtlb,\,\slidesfttrtlc,\,0,\,0\right).
        \end{gather*}
    \tikzexternalenable
    \caption{The inverse of the map in Figure~\ref{fig:slide_sf22_ltr}.}
    \label{fig:slide_sf22_rtl}
\end{figure}

\begin{lem}\label{lem:shuffle_vertex}
    The following diagram commutes up to homotopy.
    \begin{equation*}
        \shufflecross
    \end{equation*}
    Here \[\Delta_{2,2}\coloneqq\deltatwotwo\]is the $(2,2)$ shuffle braid, and $\rho_{2,2}$ is a composition of several slide maps from Proposition~\ref{prop:atomicslide}.
\end{lem}

\begin{proof}
    As in Lemma~\ref{lem:six valent}, we check that the composition $\rho_{2,2}\circ\fourv\circ\rho_{2,2}^{-1}$ from the top left term to the top right term is the same as the map $\fourva$, without worrying about the possible homotopy. The movie realizing this composition is depicted in Figure~\ref{fig:movecross}. In this figure, \[\Delta'_{2,2}\coloneqq\deltatwotwoa\]is another $(2,2)$ shuffle braid, related to $\Delta_{2,2}$ by an isotopy.

    \begin{figure}[hbtp]
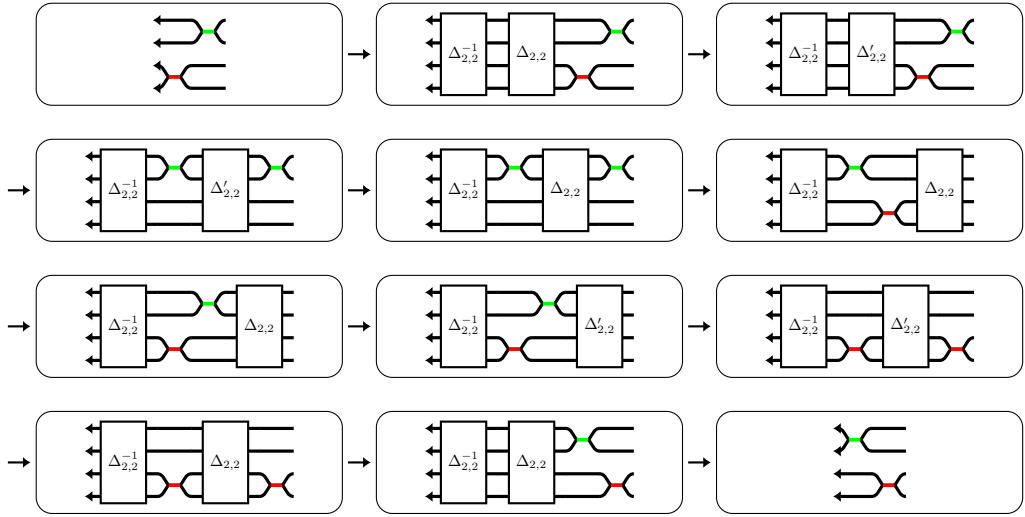

        \centering
        \[\movecross\]
        \caption{Conjugating the four-valent vertex and the $(2,2)$ shuffle braid by the half twist.}
        \label{fig:movecross}
    \end{figure}

    There are four types of moves in Figure~\ref{fig:movecross}:\begin{itemize}
        \item Reidemeister II moves (Frame 1 to 2, and 11 to 12).

        \item Isotopies between $\Delta_{2,2}$ and $\Delta_{2,2}'$ (Frame 2 to 3, 4 to 5, 7 to 8, and 9 to 10).

        \item The four-valent vertex (Frame 6 to 7).

        \item The moves that slide two strands through thick edges (Frame 3 to 4, 5 to 6, 8 to 9, and 10 to 11).
    \end{itemize}
    The induced maps of Reidemeister II moves have been described in Proposition~\ref{prop:RIImove}. The induced maps of the isotopy between $\Delta_{2,2}$ and $\Delta_{2,2}'$ are given in Figures~\ref{fig:slide_sf22_twotypes} and~\ref{fig:slide_sf22_twotypesi}. Moves of the last type are compositions of two maps described in Proposition~\ref{prop:atomicslide}; we have written them down explicitly in Proposition~\ref{prop:slide22shuffle}.

    \begin{figure}[hbtp]
    \centering
    \tikzexternaldisable
    \begin{gather*}
       \vcenter{\hbox{\includegraphics{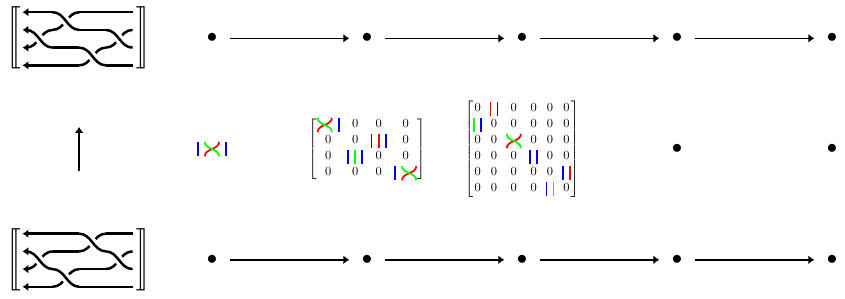}}}
        \end{gather*}
    \tikzexternalenable
    \caption{The map induced by an isotopy between $\Delta_{2,2}$ and $\Delta_{2,2}'$; parts irrelevant to later computation are omitted.}
    \label{fig:slide_sf22_twotypes}
\end{figure}

\begin{figure}[hbtp]
    \centering
    \tikzexternaldisable
    \begin{gather*}
       \vcenter{\hbox{\includegraphics{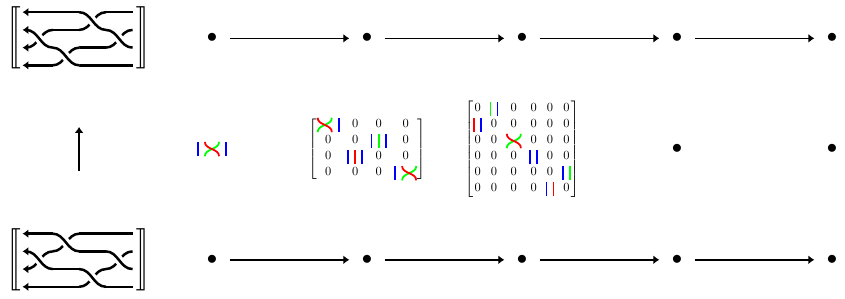}}}
        \end{gather*}
    \tikzexternalenable
    \caption{The inverse of the map in Figure~\ref{fig:slide_sf22_twotypes}.}
    \label{fig:slide_sf22_twotypesi}
\end{figure}

    As in the situation of Lemma~\ref{lem:six valent}, four crossings in $\Delta_{2,2}^{-1}$ do not participate in the movie from Frame 2 to Frame 11. By Proposition~\ref{prop:RIImove}, a resolution of eight crossings in the second frame must be \textit{palindromic} to support a nonzero image from Frame 1. Moreover, the effect of the movie from the first frame to the second and from the second-to-last frame to the last one is to take a ``partial trace'' on all but the two rightmost strands in Soergel calculus diagrams, up to sign.

    For the chain complexes associated to Frame 2 to Frame 11, we define an auxiliary grading $h_4\in\{0,1,2,3,4\}$ as the number of $1$-resolutions when resolving four crossings in $\Delta_{2,2}$ or $\Delta_{2,2}'$. It is clear from Proposition~\ref{prop:atomicslide} that the map induced by movie from Frame 2 to Frame 11 preserves the $h_4$ grading. Denote the map induced by the movie in Figure~\ref{fig:movecross} in $h_4$ grading $i$ by $f_i$ ($i=0,\,1,\,2,\,3,\,4$). We claim that all $f_i$ but $f_2$ vanish, whereas $f_2=\fourva$. It is easy to see from Proposition~\ref{prop:slide22shuffle} that $f_3$ and $f_4$ vanish. It remains to compute $f_0$, $f_1$, and $f_2$; this is done in Lemmas~\ref{lem:h40},~\ref{lem:h41}, and~\ref{lem:h42}, respectively. Therefore, the map induced by the movie in Figure~\ref{fig:movecross} is given by \[f_0+f_1+f_2+f_3+f_4=\fourva,\]as desired.
\end{proof}

\begin{lem}\label{lem:h40}
    The $h_4=0$ piece of the map induced by the movie in Figure~\ref{fig:movecross} is given by $f_0=0$.
\end{lem}

\begin{proof}
    From Frame 2 to Frame 11 in Figure~\ref{fig:movecross}, the map induced by the movie is a composition of $9$ maps. In the $h_4=0$ piece, each map is a single-term map: either a composition of two six-valent vertices as in Figures~\ref{fig:slide_sf22_ltr} and~\ref{fig:slide_sf22_rtl}, or a single four-valent vertex. As in the computation in Lemma~\ref{lem:h3=0}, we can isotope the diagram into a symmetric form and simplify it using a parabolic relation (cf. \cite[(4.27)]{stroppel2024braiding}), as depicted in Figure~\ref{fig:hugestar_four_symmetric}. The rightmost diagram in Figure~\ref{fig:hugestar_four_symmetric} can be further simplified by expanding the red and green bubbles using \cite[(4.25)]{stroppel2024braiding}. Again, as in the situation of Lemma~\ref{lem:h3=0}, there are $2^4=16$ terms in the expansion. Modulo the $C_2\times C_2$ symmetry coming from the $180^\circ$ rotation and color swapping, there are $6$ cases to check, and it turns out that all of them vanish (this is easier than the check in Lemma~\ref{lem:h3=0}). Therefore, we have $f_0=0$.

     \begin{figure}[hbtp]
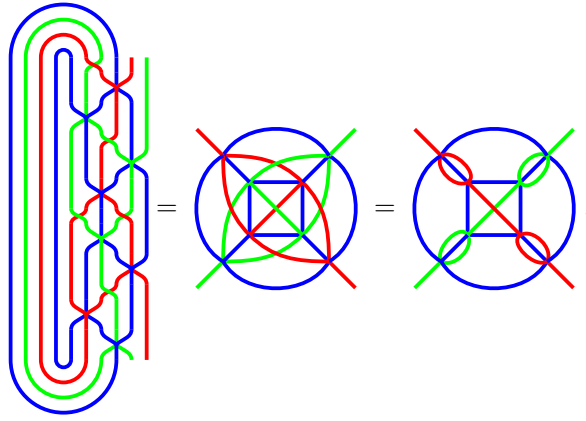

        \centering
        \begin{gather*}
            \hugestarfour=\hugestarfoursymmetric=\hugestarfoursymmetrica.
        \end{gather*}
        \caption{Isotope the $h_4=0$ piece of the map to a symmetric position and simplify using a parabolic relation.}
        \label{fig:hugestar_four_symmetric}
    \end{figure}

\end{proof}

\begin{lem}\label{lem:h41}
    The $h_4=1$ piece of the map induced by the movie in Figure~\ref{fig:movecross} is given by $f_1=0$.
\end{lem}

\begin{proof}
    Denote the $4\times 4$ matrix representing the $h_4=1$ piece of the map depicted in Figure~\ref{fig:slide_sf22_ltr} by $A_1$, the matrix obtained by swapping red and green colors of entries in $A_1$ by $B_1$, and the matrix representing the $h_4=1$ piece of the map depicted in Figure~\ref{fig:slide_sf22_twotypes} by $X_1$. Then the map $f_1$ is given by \[f_1=\operatorname{tr}\left(A_1^{-1}\cdot X_1^{-1}\cdot B_1^{-1}\cdot X_1\cdot \fourv\cdot  A_1\cdot X_1^{-1}\cdot B_1\cdot X_1\right),\]where $\operatorname{tr}$ refers to taking the trace of the matrix and closing the three leftmost strands up. Note that here we slightly abuse the notation: the entries in matrices $A_1$, $B_1$, and $X_1$ have four strands, whereas our actual computation has five strands. The missing strand comes from the thick edge that is not involved in the local computation; the actual matrices are obtained by tensoring the entries in $A_1$, $B_1$, and $X_1$ with the identity map on the missing strand. 

    It turns out that the products $A_1^{-1}\cdot X_1^{-1}\cdot B_1^{-1}\cdot X_1$ and $A_1\cdot X_1^{-1}\cdot B_1\cdot X_1$ both have single-term entries, so the trace of the whole composition has $16$ terms. This is similar to the situation in Lemma~\ref{lem:h3=1}, and we check that all $16$ terms vanish. The $(1,1)$ entry of the composition is given by \[\hfouroneoneonea+\hfouroneoneoneb+\hfouroneoneonec+\hfouroneoneoned=0.\]The $(2,2)$ entry of the composition is given by \[\hfouronetwotwoa+\hfouronetwotwob+\hfouronetwotwoc+\hfouronetwotwod=0.\]The $(3,3)$ entry of the composition is given by \[\hfouronethreethreea+\hfouronethreethreeb+\hfouronethreethreec+\hfouronethreethreed=0.\]The $(4,4)$ entry of the composition is given by \[\hfouronefourfoura+\hfouronefourfourb+\hfouronefourfourc+\hfouronefourfourd=0.\]Here most of the terms vanish because one can trace closed curves that bound disks inside them, and the most nontrivial one is the last term in the $(4,4)$ entry, which is simplified by isotoping to a symmetric position and using a parabolic relation (cf. \cite[(4.27)]{stroppel2024braiding}) similar to Figure~\ref{fig:hugestar_four_symmetric}. 
\end{proof}

\begin{lem}\label{lem:h42}
    The $h_4=2$ piece of the map induced by the movie in Figure~\ref{fig:movecross} is given by $f_2=\fourva$.
\end{lem}

\begin{proof}
    Denote the $6\times 6$ matrix representing the $h_4=2$ piece of the map depicted in Figure~\ref{fig:slide_sf22_ltr} by $A_2$, the matrix obtained by swapping red and green colors of entries in $A_2$ by $B_2$, and the matrix representing the $h_4=2$ piece of the map depicted in Figure~\ref{fig:slide_sf22_twotypes} by $X_2$. Then the map $f_2$ is given by \[f_2=\operatorname{tr}\left(A_2^{-1}\cdot X_2^{-1}\cdot B_2^{-1}\cdot X_2\cdot \fourv\cdot  A_2\cdot X_2^{-1}\cdot B_2\cdot X_2\right),\]where $\operatorname{tr}$ refers to taking the trace of the matrix and closing the two leftmost strands up. We again abuse the notation as in Lemma~\ref{lem:h41}. 

    Again, the products $A_2^{-1}\cdot X_2^{-1}\cdot B_2^{-1}\cdot X_2$ and $A_2\cdot X_2^{-1}\cdot B_2\cdot X_2$ both have single-term entries. Moreover, many terms in these matrices vanish: \[A_2^{-1}\cdot X_2^{-1}\cdot B_2^{-1}\cdot X_2=\begin{bmatrix}
        0 & 0 & 0 & 0 & 0 & 0\\
        \bullet & 0 & \bullet  & \bullet  & 0 & \bullet \\
       \bullet & 0 & {\color{red}\bullet}  & {\color{red}\bullet}  & 0 & \bullet \\
        \bullet & 0 & {\color{red}\bullet} & {\color{red}\bullet}  & 0 & \bullet \\
       \bullet & 0 & \bullet  & \bullet  & 0 & \bullet \\
        0 & 0 & 0 & 0 & 0 & 0
    \end{bmatrix},\ A_2\cdot X_2^{-1}\cdot B_2\cdot X_2=\begin{bmatrix}
        0 & 0 & 0 & 0 & 0 & 0\\
        \bullet & 0 & \bullet  & \bullet  & 0 & \bullet \\
       \bullet & 0 & {\color{red}\bullet}  & {\color{red}\bullet}  & 0 & \bullet \\
        \bullet & 0 & {\color{red}\bullet} & {\color{red}\bullet}  & 0 & \bullet \\
       \bullet & 0 & \bullet  & \bullet  & 0 & \bullet \\
        0 & 0 & 0 & 0 & 0 & 0
    \end{bmatrix}.\]Here bullets indicate entries that are possibly nonzero, and red bullets indicate entries that possibly contribute to the trace. All diagonal entries of the composition are zero except for the $(3,3)$ and $(4,4)$ entries. The $(3,3)$ entry is given by \[\hfourtwothreethreea+\hfourtwothreethreeb=\fourva+0=\fourva,\]and the $(4,4)$ entry is given by \[\hfourtwofourfoura+\hfourtwofourfourb=0.\]In conclusion, we have $f_2=\fourva+0=\fourva$.
\end{proof}

\begin{lem}\label{lem:double_fork_twist}
    The following diagram commutes up to homotopy.
    \begin{equation*}
        \doubleforkcross
    \end{equation*}
\end{lem}

\begin{proof}
    This is essentially a consequence of the perfect commutativity of diagrams with distant colors (cf. \cite[(4.26)]{stroppel2024braiding}). To be precise, one can check that the two compositions of maps from the bottom left to the top right are both equal to \[\Figdoubletwist.\]
\end{proof}

\bibliographystyle{amsalpha}
\bibliography{ref}

\end{document}